\documentclass[12pt,reqno]{amsart}

\usepackage{amsthm}
\usepackage{amscd}
\usepackage{amsfonts}
\usepackage{amssymb}
\usepackage{amsgen}
\usepackage{amsmath}
\usepackage{amsopn}
\usepackage{verbatim}
\usepackage{xypic}
\usepackage{xspace}
\usepackage{multicol}
\usepackage{url}
\usepackage{upref}
\usepackage{hyperref}
\usepackage{graphicx}
\usepackage{color}
\usepackage{colordvi}
\usepackage[active]{srcltx} 

\theoremstyle{plain}
\newtheorem{thm}{Theorem}[section]
\newtheorem{lem}[thm]{Lemma}
\newtheorem{prop}[thm]{Proposition}
\newtheorem{cor}[thm]{Corollary}
\newtheorem{conj}[thm]{Conjecture}

\theoremstyle{definition}

\theoremstyle{remark}

\newtheorem{eg}[thm]{Example}

 \font\cyr=wncyr10
 \newcommand{\nc}{\newcommand}
\nc{\Cyc}[1]{\underset{#1}{\rm Cyc}\,}
\nc{\Sym}[1]{\underset{#1}{\rm Sym}\,}

\renewcommand\Re{{\rm Re}}

\nc{\per}[1]{\underset{#1}{\boldsymbol \pi}\,}

 \nc{\MT}{{\rm MT}}
 \nc{\XX}{{X}}
 \nc{\gF}{{\varPhi}}
 \nc{\ot}{\otimes}

 \nc{\wht}{\widehat}
 \nc{\bwg}{{\bigwedge}}
 \nc{\wg}{{\wedge}}
 \nc{\mmu}{{\boldsymbol{\mu}}}
 \nc{\mal}{{{\scriptstyle \maltese}}}
 \nc{\fA}{{\mathfrak A}}
 \nc{\HH}{{\mathfrak H}}
 \nc{\ra}{\rightarrow}
 \nc{\ors}{{\bfs}}
 \nc{\orr}{{\bfr}}
 \nc{\os}{{\overset}}
 \nc{\G}{{\mathbb G}}
 \nc{\F}{{\mathbb F}}
 \nc{\Z}{{\mathbb Z}}
 \nc{\R}{{\mathbb R}}
 \nc{\N}{{\mathbb N}}
 \nc{\ZN}{{\mathbb Z_{\ge 0}}}
 \nc{\Q}{{\mathbb Q}}
 \nc{\C}{{\mathbb C}}
 \nc{\CP}{{\mathbb{CP}}}
 \nc{\Cnn}{{\mathbb C}_{\ge 0}}
 \nc{\Cp}{{\mathbb C}_{>0}}
 \nc{\MPV}{{\mathcal{MPV}}}
 
 \nc{\tB}{{\tilde B}}

 \nc{\suf}{{\ast\,}}
 \nc{\sufq}{{\ast_q\,}}
 \nc{\gam}{{\gamma}}
 \nc{\gG}{{\Gamma}}
 \nc{\om}{{\omega}}
 \nc{\vep}{{\varepsilon}}
 \nc{\ga}{{\alpha}}
 \nc{\gl}{{\lambda}}
 \nc{\gb}{{\beta}}
 \nc{\gf}{{\varphi}}
 \nc{\gd}{{\delta}}
 \nc{\orgd}{{\vec \gd\,}}
 \nc{\gs}{{\sigma}}
 \nc{\gth}{{\theta}}
 \nc{\gS}{{\Sigma}}

 \nc{\gk}{{\kappa}}
  \nc{\gz}{{\zeta}}
 \nc{\tgz}{{\tilde{\zeta}}}
 \nc{\gO}{{\Omega}}
 \nc{\sif}{{\mathcal S}}
 \nc{\gt}{{\tau}}
 \nc{\Lra}{\Longrightarrow}
 \nc{\lra}{\longrightarrow}
 \nc{\lmaps}{\longmapsto}
 \nc{\fS}{{\mathfrak S}}
 \nc{\DD}{{\mathfrak D}}
 \nc{\Llra}{\Longleftrightarrow}
 \nc{\ol}{\overline}
 \nc{\ola}{\overleftarrow}
 \nc{\lms}{\longmapsto}
 \nc{\cv}{{{\mathsf c}{\mathsf v}}}
 \nc{\zq}{{\zeta_q}}
 \nc\qup{{q\uparrow 1}}
 \nc{\us}{\underset}
 \nc{\tn}{{\tilde{n}}}
 \nc{\gD}{{\Delta}}
 \nc{\bi}{{\bf i}}
 \nc{\bfone}{{\bf 1}}

 \nc{\bfa}{{\bf a}}
 \nc{\bfb}{{\bf b}}
 \nc{\bfc}{{\bf c}}
 \nc{\bfd}{{\bf d}}
 \nc{\bfe}{{\bf e}}
 \nc{\bff}{{\bf f}}
 \nc{\bfg}{{\bf g}}
 \nc{\bfi}{{\bf i}}
 \nc{\bfj}{{\bf j}}

 \nc{\bfn}{{\bf n}}
 \nc{\bfl}{{\bf l}}
 \nc{\bfk}{{\bf k}}
 \nc{\bfm}{{\bf m}}
 \nc{\bfo}{{\bf o}}
 \nc{\bfp}{{\bf p}}
 \nc{\bfq}{{\bf q}}
 \nc{\bfr}{{\bf r}}
 \nc{\bfs}{{\bf s}}
 \nc{\bft}{{\bf t}}
 \nc{\bfu}{{\bf u}}
 \nc{\bfv}{{\bf v}}
 \nc{\bfw}{{\bf w}}
 \nc{\bfx}{{\bf x}}
 \nc{\bfy}{{\bf y}}
 \nc{\bfz}{{\bf z}}
 \nc{\bfB}{{\bf B}}
 \nc{\bfP}{{\bf P}}
 \nc{\bfQ}{{\bf Q}}
 \nc{\bfY}{{\bf Y}}
 \nc{\bfgb}{{\boldsymbol \gb}}
 \nc{\bfga}{{\boldsymbol \ga}}
 \nc{\bfrho}{{\boldsymbol \rho}}
 \nc{\bfchi}{{\boldsymbol \chi}}
 \nc{\QX}{{\Q\langle \bfX\rangle}}
 \nc{\QY}{{\Q\langle \bfY\rangle}}
 \nc{\CX}{{\C\langle \bfX\rangle}}
 \nc{\CY}{{\C\langle \bfY\rangle}}
 \nc{\QXX}{{\Q\langle\!\langle \bfX\rangle\!\rangle}}
 \nc{\QYY}{{\Q\langle\!\langle \bfY\rangle\!\rangle}}
 \nc{\CXX}{{\C\langle\!\langle \bfX\rangle\!\rangle}}
 \nc{\CYY}{{\C\langle\!\langle \bfY\rangle\!\rangle}}

 \nc{\bbA}{{\mathbb A}}
 \nc{\bbB}{{\mathbb B}}
 \nc{\bbC}{{\mathbb C}}
 \nc{\bbD}{{\mathbb D}}
 \nc{\bbE}{{\mathbb E}}
 \nc{\bbF}{{\mathbb F}}
 \nc{\bbG}{{\mathbb G}}
 \nc{\bbH}{{\mathbb H}}
 \nc{\bbI}{{\mathbb I}}
 \nc{\bbJ}{{\mathbb J}}
 \nc{\bbK}{{\mathbb K}}
 \nc{\bbL}{{\mathbb L}}
 \nc{\bbM}{{\mathbb M}}
 \nc{\bbN}{{\mathbb N}}
 \nc{\bbO}{{\mathbb O}}
 \nc{\bbP}{{\mathbb P}}
 \nc{\bbQ}{{\mathbb Q}}
 \nc{\bbR}{{\mathbb R}}
 \nc{\bbS}{{\mathbb S}}
 \nc{\bbT}{{\mathbb T}}
 \nc{\bbU}{{\mathbb U}}
 \nc{\bbV}{{\mathbb V}}
 \nc{\bbW}{{\mathbb W}}
 \nc{\bbX}{{\mathbb X}}
 \nc{\bbY}{{\mathbb Y}}
 \nc{\bbZ}{{\mathbb Z}}
 \nc{\bba}{{\mathbb a}}
 \nc{\bbb}{{\mathbb b}}
 \nc{\bbc}{{\mathbb c}}
 \nc{\bbd}{{\mathbb d}}
 \nc{\bbe}{{\mathbb e}}
 \nc{\bbf}{{\mathbb f}}
 \nc{\bbg}{{\mathbb g}}
 \nc{\bbh}{{\mathbb h}}
 \nc{\bbi}{{\mathbb i}}
 \nc{\bbk}{{\mathbb k}}
 \nc{\bbl}{{\mathbb l}}
 \nc{\bbm}{{\mathbb m}}
 \nc{\bbn}{{\mathbb n}}
 \nc{\bbo}{{\mathbb o}}
 \nc{\bbp}{{\mathbb p}}
 \nc{\bbq}{{\mathbb q}}
 \nc{\bbr}{{\mathbb r}}
 \nc{\bbs}{{\mathbb s}}
 \nc{\bbt}{{\mathbb t}}
 \nc{\bbu}{{\mathbb u}}
 \nc{\bbv}{{\mathbb v}}
 \nc{\bbw}{{\mathbb w}}
 \nc{\bbx}{{\mathbb x}}
 \nc{\bby}{{\mathbb y}}
 \nc{\bbz}{{\mathbb z}}

 \nc{\MZV}{{\mathcal{MZV}}}
 \nc{\calA}{{\mathcal A}}
 \nc{\calB}{{\mathcal B}}
 \nc{\calC}{{\mathcal C}}
 \nc{\calD}{{\mathcal D}}
 \nc{\calE}{{\mathcal E}}
 \nc{\calF}{{\mathcal F}}
 \nc{\calG}{{\mathcal G}}
 \nc{\calH}{{\mathcal H}}
 \nc{\calI}{{\mathcal I}}
 \nc{\calJ}{{\mathcal J}}
 \nc{\calK}{{\mathcal K}}
 \nc{\calL}{{\mathcal L}}
 \nc{\calM}{{\mathcal M}}
 \nc{\calN}{{\mathcal N}}
 \nc{\calO}{{\mathcal O}}
 \nc{\calP}{{\mathcal P}}
 \nc{\calQ}{{\mathcal Q}}
 \nc{\calR}{{\mathcal R}}
 \nc{\calS}{{\mathcal S}}
 \nc{\calT}{{\mathcal T}}
 \nc{\calU}{{\mathcal U}}
 \nc{\calV}{{\mathcal V}}
 \nc{\calW}{{\mathcal W}}
 \nc{\calX}{{\mathcal X}}
 \nc{\calY}{{\mathcal Y}}
 \nc{\calZ}{{\mathcal Z}}
  \nc{\cala}{{\mathcal a}}
 \nc{\calb}{{\mathcal b}}
 \nc{\calc}{{\mathcal c}}
 \nc{\cald}{{\mathcal d}}
 \nc{\cale}{{\mathcal e}}
 \nc{\calf}{{\mathcal f}}
 \nc{\calg}{{\mathcal g}}
 \nc{\calh}{{\mathcal h}}
 \nc{\cali}{{\mathcal i}}
 \nc{\calj}{{\mathcal j}}
 \nc{\calk}{{\mathcal k}}
 \nc{\call}{{\mathcal l}}
 \nc{\calm}{{\mathcal m}}
 \nc{\caln}{{\mathcal n}}
 \nc{\calo}{{\mathcal o}}
 \nc{\calp}{{\mathsf p}}
 \nc{\calq}{{\mathcal q}}
 \nc{\calr}{{\mathcal r}}
 \nc{\cals}{{\mathcal s}}
 \nc{\calt}{{\mathcal t}}
 \nc{\calu}{{\mathcal u}}
 \nc{\calv}{{\mathcal v}}
 \nc{\calw}{{\mathcal w}}
 \nc{\calx}{{\mathcal x}}
 \nc{\caly}{{\mathcal y}}
 \nc{\calz}{{\mathcal z}}

 \nc{\frakA}{{\mathfrak A}}
 \nc{\frakB}{{\mathfrak B}}
 \nc{\frakC}{{\mathfrak C}}
 \nc{\frakD}{{\mathfrak D}}
 \nc{\frakE}{{\mathfrak E}}
 \nc{\frakF}{{\mathfrak F}}
 \nc{\frakG}{{\mathfrak G}}
 \nc{\frakH}{{\mathfrak H}}
 \nc{\frakI}{{\mathfrak I}}
 \nc{\frakJ}{{\mathfrak J}}
 \nc{\frakK}{{\mathfrak K}}
 \nc{\frakL}{{\mathfrak L}}
 \nc{\frakM}{{\mathfrak M}}
 \nc{\frakN}{{\mathfrak N}}
 \nc{\frakO}{{\mathfrak O}}
 \nc{\frakP}{{\mathfrak P}}
 \nc{\frakQ}{{\mathfrak Q}}
 \nc{\frakR}{{\mathfrak R}}
 \nc{\frakS}{{\mathfrak S}}
 \nc{\frakT}{{\mathfrak T}}
 \nc{\frakU}{{\mathfrak U}}
 \nc{\frakV}{{\mathfrak V}}
 \nc{\frakW}{{\mathfrak W}}
 \nc{\frakX}{{\mathfrak X}}
 \nc{\frakY}{{\mathfrak Y}}
 \nc{\frakZ}{{\mathfrak Z}}
 \nc{\fraka}{{\mathfrak a}}
 \nc{\frakb}{{\mathfrak b}}
 \nc{\frakc}{{\mathfrak c}}
 \nc{\frakd}{{\mathfrak d}}
 \nc{\frake}{{\mathfrak e}}
 \nc{\frakf}{{\mathfrak f}}
 \nc{\frakg}{{\mathfrak g}}
 \nc{\frakh}{{\mathfrak h}}
 \nc{\fraki}{{\mathfrak i}}
 \nc{\frakj}{{\mathfrak j}}
 \nc{\frakk}{{\mathfrak k}}
 \nc{\frakl}{{\mathfrak l}}
 \nc{\frakm}{{\mathfrak m}}
 \nc{\frakn}{{\mathfrak n}}
 \nc{\frako}{{\mathfrak o}}
 \nc{\frakp}{{\mathfrak p}}
 \nc{\frakq}{{\mathfrak q}}
 \nc{\frakr}{{\mathfrak r}}
 \nc{\fraks}{{\mathfrak s}}
 \nc{\frakt}{{\mathfrak t}}
 \nc{\fraku}{{\mathfrak u}}
 \nc{\frakv}{{\mathfrak v}}
 \nc{\frakw}{{\mathfrak w}}
 \nc{\frakx}{{\mathfrak x}}
 \nc{\fraky}{{\mathfrak y}}
 \nc{\frakz}{{\mathfrak z}}
 \nc{\so}{{\mathfrak so}}
 \nc{\sa}{{\mbox{{\scriptsize \cyr x}}}}
 \nc{\slfour}{{\mathfrak sl}_4}
 \nc{\one}{{\bf 1}}
 \nc{\zero}{{\bf 0}}
 \nc{\Qxy}{\Q\langle x,y\rangle}

\nc{\li}[1]{\textcolor{red}{\tt Li:#1}}

\begin{document}

\title{Families of weighted sum formulas for\\ multiple zeta values}

\author{Li Guo}
\address{
Department of Mathematics and Computer Science,
Rutgers University,
Newark, NJ 07102, USA}
\email{liguo@rutgers.edu}

\author{Peng Lei}
\address{Department of Mathematics,
    Lanzhou University,
    Lanzhou, Gansu 730000, China}
\email{leip@lzu.edu.cn}

\author{Jianqiang Zhao}
\address{Department of Mathematics,
Eckerd College, St. Petersburg, FL 33711, USA}
\email{zhaoj@ekcerd.edu}

\date{}

\begin{abstract}
Euler's sum formula and its multi-variable and weighted generalizations form
a large class of the identities of multiple zeta values.
In this paper we prove a family of identities involving Bernoulli numbers and
apply them to obtain infinitely many weighted sum formulas for double zeta values
and triple zeta values where the weight coefficients are given by symmetric polynomials.
\end{abstract}


\maketitle

\tableofcontents

\allowdisplaybreaks

\section{Introduction}

Multiple zeta functions are multiple variable generalizations of the Riemann zeta function.
For fixed positive integer $d$ and $d$-tuple of complex numbers $\bfs=(s_1,\dots, s_d)$,
the multiple zeta function is defined by
\begin{equation}\label{equ:MZFdefn}
\zeta(\bfs)=\sum_{ k_1>\dots>k_d>0} k_1^{-s_1}\cdots k_d^{-s_d}
\end{equation}
where $\bfs$ satisfies $\Re(s_1+\dots+s_j)>j$ for all $j=1,\dots,d$. The number $d$ is
called the \emph{depth} (or \emph{length}) and
$s_1+\dots+s_d$ the \emph{weight}, denoted by $|\bfs|.$ Their convergent
special values at positive integers are called multiple zeta values. These values
can be traced back to a series of
correspondences between Leonhard Euler and Christian Goldbach \cite{JuskevicWi1965}.
On the Christmas Eve of 1742, with different notation Goldbach wrote down some special cases
of the following infinite sum on a letter to Euler:
\begin{equation*}
\sum_{a\ge b\ge 1} \frac{1}{a^m b^n},
\end{equation*}
where $m$ and $n$ are positive integer. Using our notation this is $\zeta(m,n)+\zeta(m+n)$
where $\zeta(m,n)$ is a double zeta value (DZV for short). Later Euler discovered the following sum formula
\begin{equation}\label{eq:sum}
\sum_{k=2}^{n-1} \zeta(k,n-k) = \zeta(n),\quad n\geq 3,
\end{equation}
and decomposition formula
\begin{equation}\label{equ:EulerDecompositionFormula}
  2\zeta(n,1)=n\zeta(n+1)-\sum_{i=1}^{n-2} \zeta(n-i)\zeta(i+1), \quad n\geq 2.
\end{equation}

There are many generalizations and variations of the sum formula in the literature.
Ohno and Zudilin~\cite{OhnoZu2008} proved a weighted form of Euler's sum
formula
\begin{equation}\label{equ:WtSumMZV}
\sum_{k=2}^{w-1} 2^k \zeta(k,w-k) = (w+1)\zeta (w), \quad w\ge 3.
\end{equation}
Later this was generalized by Guo and Xie \cite{GuoXi2009} to arbitrary depths.
During their study of DZVs and modular forms
Gangl et al.\ \cite{GKZ2006} made the following discovery:
For all $n\ge 2$ we have
\begin{align} \label{equ:GKZDblZetaEven}
    \sum_{k=1}^{n-1} \zeta(2k,2n-2k)=&\frac34\zeta(2n),  \\
    \sum_{k=1}^{n-1} \zeta(2k+1,2n-2k-1)=&\frac14\zeta(2n). \label{equ:GKZDblZetaOdd}
\end{align}
Recently, Hoffman \cite{Hoffman2012} extended Eq.~\eqref{equ:GKZDblZetaEven} to arbitrary depths.
Of course these can also be regarded as weighted sum formulas.
Some more complicated identities in depth two can be found in \cite{Nakamura2009}
\begin{align} \label{equ:NakamuraDblZeta1}
  \sum_{k=1}^{n-1} (4^k+4^{n-k}) \zeta(2k,2n-2k)=\left(n+\frac43+\frac{4^n}6\right)\zeta(2n),\\
\label{equ:NakamuraDblZeta2}
  \sum_{k=2}^{n-2} (2k-1)(2n-2k-1)  \zeta(2k,2n-2k)=\frac34 (n-3)\zeta(2n).
\end{align}
Nakamura's idea to prove Eq.~\eqref{equ:NakamuraDblZeta2} is to show
the following identity of Bernoulli numbers:
\begin{equation}\label{equ:CorrectEie}
 6\sum_{i,j\ge 4, i+j=k} (i-1)(j-1) B_iB_j\binom{k}{i}
=- (k-1)(k^2-5k-6) B_k,
\end{equation}
where $B_j$  is a  Bernoulli number with generating function
\begin{equation*}
     \frac{t}{e^t-1}=\sum_{j=0}^\infty B_j\frac{t^j}{j!}.
\end{equation*}
Eq.~\eqref{equ:CorrectEie} quickly leads to
\begin{equation}\label{equ:Eier=1Result}
 6\sum_{k=2}^{n-2} (2k-1)(2n-2k-1) \zeta(2k)\zeta(2n-2k)
=(n-3)(4n^2-1) \zeta(2n)
\end{equation}
by Euler's famous evaluation
\begin{equation}\label{equ:EulerFamousID}
 \zeta(2n) =-\frac{B_{2n}}{2(2n)!}(2\pi i)^{2n},\quad \text{and}\quad \zeta(1-2n) =-\frac{B_{2n}}{2n}.
\end{equation}

The identity in Eq.~\eqref{equ:CorrectEie} relating Bernoulli numbers
have been obtained by a few diverse methods.
Rademacher \cite[p.~121]{Rademacher1973} derived
it as a consequence of an identity
among Eisenstein series of different weights. Shimura essential
did the same in his book \cite[(11.10)]{Shimura2007}.
Eie \cite{Eie1996} proved it using the zeta function
associated with some polynomials. We use his method in this paper to obtain
infinitely many families of Bernoulli number identities
similar to Eq.~\eqref{equ:CorrectEie} in Section~\ref{sec:bern}. These identities lead to infinitely many different
weighted sum formulas for double and triple zeta values with symmetric polynomial coefficients. We will consider the case of double zeta values in Section~\ref{sec:double} and the case of triple zeta values in Section~\ref{sec:triple}. We end the paper with a conjecture for the general case.

\section{Weighted sum of products of Bernoulli numbers}
\label{sec:bern}
In this section we shall prove a sum formula for products $B_{2j}B_{2k}$
for fixed $j+k$ with some weight coefficients.
This in turn will lead to a sum formula for products $\zeta(2j)\zeta(2k)$
with weight coefficients given by arbitrary polynomials in $j$ and $k$.
First we define a zeta function that will be useful in both depth two and depth three cases:
\begin{equation*}
Z_2(s;r_1,r_2)=\sum_{m_1,m_2=1}^\infty m_1^{r_1}
    m_2^{r_2}(m_1+m_2)^{-s}.
\end{equation*}
The basic idea of Eie in \cite{Eie1996,Eie2009} is to compute the special value of this function when $s$ is
some appropriate negative integer using two different methods. By comparing
the two expressions one can derive an identity of Bernoulli numbers which yields the
desired identity of multiple zeta values.

First we want to find a useful expression of $Z_2(s;r_1,r_2)$. We have
\begin{align}
Z_2(s;r_1,r_2)=&\sum_{m=1}^\infty\sum_{m_1=1}^{m-1}m_1^{r_1}(m-m_1)^{r_2}m^{-s} \notag\\
=&\sum_{m=1}^\infty\sum_{m_1=1}^{m-1}m_1^{r_1}\sum_{i=0}^{r_2}\binom{r_2}{i}(-1)^i
     m^{r_2-i}m_1^im^{-s} \notag\\
=&\sum_{m=1}^\infty\sum_{i=0}^{r_2}(-1)^i\binom{r_2}{i}
    \left(\sum_{k=0}^{r_1+i}\binom{r_1+i+1}{k}\frac{(-1)^kB_k}{r_1+i+1}m^{r_1+i+1-k}-m^{r_1+i}\right)\cdot m^{r_2-i-s}\notag\\
=&\sum_{i=0}^{r_2}(-1)^i\binom{r_2}{i}\sum_{k=0}^{r_1+i}\binom{r_1+i+1}{k} \frac{(-1)^kB_k}{r_1+i+1}f_{k}(s;r_1,r_2)-\gd_{r_2,0}\zeta(s-r_1-r_2),\label{equ:Z2step1}
\end{align}
where $f_{k}(s;r_1,r_2)=\zeta(s+k-r_1-r_2-1)$.
To simplify this further we need the following combinatorial lemmas.
\begin{lem}\label{lem:Combin1}
Let $k$, $r_1$, and $r_2$ be nonnegative integers. Then
$$
\sum_{i=0}^{r_2}(-1)^i\binom{r_2}{i}\binom{r_1+i}{k-1}
 =(-1)^{r_2}\binom{r_1}{k-r_2-1}.
$$
\end{lem}
\begin{proof} We have
{\allowdisplaybreaks
\begin{align*}
 & \sum_{k=1}^{r_1+r_2+1}\sum_{i=0}^{r_2}(-1)^i\binom{r_2}{i}\binom{r_1+i}{k-1}x^k\\
=&\sum_{i=0}^{r_2}(-1)^i\binom{r_2}{i}\sum_{k=0}^{r_1+r_2}\binom{r_1+i}{k} x^{k+1}
=\sum_{i=0}^{r_2}(-1)^i\binom{r_2}{i}x(1+x)^{r_1+i}\\
=&x(1+x)^{r_1}(1-(1+x))^{r_2}
=\sum_{l=0}^{r_1}\binom{r_1}{l}(-1)^{r_2}x^{r_2+l+1}
=\sum_{k=0}^{r_1+r_2+1}(-1)^{r_2}\binom{r_1}{k-r_2-1}x^k.
\end{align*}}
Comparing the coefficients we obtain the lemma immediately.
\end{proof}

\begin{lem}\label{lem:Combin2}
Let $r_1$ and $r_2$ be two nonnegative integers. Then
\begin{equation*}
\sum_{i=0}^{r_2}(-1)^i\binom{r_2}{i}\frac{1}{r_1+i+1}=\frac{ r_1!r_2!}{(r_1+r_2+1)!}.
\end{equation*}
\end{lem}
\begin{proof} Define
\begin{equation*}
F(x)=\sum_{i=0}^{r_2}(-1)^i\binom{r_2}{i}\frac{x^{r_1+i+1}}{r_1+i+1}  .
\end{equation*}
Then
\begin{equation*}
F'(x)=\sum_{i=0}^{r_2}(-1)^i\binom{r_2}{i} x^{r_1+i}=x^{r_1} (1-x)^{r_2}.
\end{equation*}
Thus using the beta function $B(a,b)=\int_0^1 x^{a-1}(1-x)^{b-1} \, dx=\gG(a)\gG(b)/\gG(a+b)$ we have
\begin{equation*}
\sum_{i=0}^{r_2}(-1)^i\binom{r_2}{i}\frac{1}{r_1+i+1}
=\int_0^1 x^{r_1} (1-x)^{r_2} \, dx=B(r_1+1,r_2+1)=\frac{ r_1!r_2!}{(r_1+r_2+1)!}
\end{equation*}
as desired.
\end{proof}
Given any function $f(x,y)$ we define
$$\Cyc{x,y} f(x,y)=f(x,y)+f(y,x).$$
\begin{prop}\label{prop:Z2general}
For all nonnegative integers $r_1$ and $r_2$ we have
\begin{equation*}
Z_2(s;r_1,r_2)=\frac{ r_1!r_2!}{(r_1+r_2+1)!} f_{0}(s;r_1,r_2)+
\Cyc{r_1,r_2} (-1)^{r_1}\sum_{k=r_1+1}^{r_1+r_2+1}\binom{r_2}{k-r_1-1}\beta_k f_{k}(s;r_1,r_2),
\end{equation*}
where $\beta_k=B_k/k$.
\end{prop}
\begin{proof}
 By Eq.~\eqref{equ:Z2step1} we get
{\allowdisplaybreaks
\begin{align*}
Z_{2}(s;r_1,r_2)=&\sum_{i=0}^{r_2}(-1)^i\binom{r_2}{i}
 \sum_{k=0}^{r_1+r_2+1}\binom{r_1+i+1}{k}\frac{(-1)^k B_k}{r_1+i+1}f_{k}(s;r_1,r_2)\\
&+\sum_{i=0}^{r_2}(-1)^{r_1}\binom{r_2}{i}\frac{B_{r_1+i+1}}{r_1+i+1}
    f_{r_1+i+1}(s;r_1,r_2)-\gd_{r_2,0}\zeta(s-r_1-r_2)\\
=&\sum_{i=0}^{r_2}(-1)^i\binom{r_2}{i}\frac{1}{r_1+i+1}f_{0}(s;r_1,r_2)
+ \sum_{i=0}^{r_2}(-1)^{r_1}\binom{r_2}{i}\beta_{r_1+i+1}f_{r_1+i+1}(s;r_1,r_2) \\
 &+\sum_{i=0}^{r_2}(-1)^i\binom{r_2}{i}\sum_{k=1}^{r_1+r_2+1}
 \binom{r_1+i}{k-1} (-1)^k \beta_k f_{k}(s;r_1,r_2)-\gd_{r_2,0}\zeta(s-r_1-r_2).
\end{align*}}
By Lemma \ref{lem:Combin1} and Lemma~\ref{lem:Combin2} we have
\begin{align*}
Z_2(s;r_1,r_2)=&\frac{ r_1!r_2!}{(r_1+r_2+1)!} f_{0}(s;r_1,r_2)+
(-1)^{r_1}\sum_{k=r_1+1}^{r_1+r_2+1}\binom{r_2}{k-r_1-1}\beta_k f_{k}(s;r_1,r_2)\\
 +&(-1)^{r_2}\sum_{k=r_2+1}^{r_1+r_2+1}\binom{r_1}{k-r_2-1} (-1)^k \beta_k f_{k}(s;r_1,r_2)-\gd_{r_2,0}\zeta(s-r_1-r_2).
\end{align*}
Observe that $(-1)^k \beta_k=\beta_k$ unless $k=1$, and this term
occurs if and only if $r_2=0$. Further $\beta_1=-\beta_1-1$. So by combining the last two parts of the above equation, we can get the result of the proposition.
\end{proof}

In order to find the weighted sum formulas for DZVs
we first consider the corresponding result for Bernoulli numbers.
\begin{thm}\label{thm:myGenBernoulliWtSumDepth2}
For all nonnegative integers $r_1$, $r_2$ and $n\ge r_1+r_2+2$ we have
\begin{align*}
\ & \sum_{k=1}^{2n-1}
    \frac{B_{k}}{k!}\frac{B_{2n-k}}{(2n-k)!} \prod_{a=1}^{r_1}(k-a)\prod_{b=1}^{r_2}(2n-k-b) \\
=& \sum_{k=r_1+1}^{2n-r_2-1}
    \frac{B_{k}}{k!}\frac{B_{2n-k}}{(2n-k)!} \prod_{a=1}^{r_1}(k-a)\prod_{b=1}^{r_2}(2n-k-b) \\
=&-r_1!r_2!\left(\binom{2n-1}{r_1+r_2+1}+\Cyc{r_1,r_2} (-1)^{r_2} \binom{2n-1}{r_1}\right)\frac{B_{2n}}{(2n)!} \\
& -r_1!r_2!\, \Cyc{r_1,r_2}(-1)^{r_1}\sum_{k=r_1+1}^{r_1+r_2+1}
\binom{2n-1-k}{r_1+r_2+1-k}\binom{k-1}{r_1}
 \frac{B_{k} }{k!}  \frac{B_{2n-k}}{(2n-k)!}.
\end{align*}
\end{thm}

\begin{proof}
Let $s=r_1+r_2+2-2n$. By Proposition \ref{prop:Z2general} we have
\begin{equation}\label{equ:Si-Siii}
Z_2(s;r_1,r_2)=\frac{ -r_1!r_2!\beta_{2n}}{(r_1+r_2+1)!}-
\Cyc{r_1,r_2}(-1)^{r_1}\sum_{k=r_1+1}^{r_1+r_2+1}\binom{r_2}{k-r_1-1}\beta_k \beta_{2n-k}.
\end{equation}
On the other hand Eie \cite{Eie1996,Eie2009} showed that functions like $Z_r(s)$ have
analytic continuations over the whole complex plane and further they are defined at
negative integers. In fact our expressions in \eqref{equ:Si-Siii} shows clearly that this can
be done using the Riemann zeta function. More importantly, Eie showed that these special
values at negative integers can also be computed using some integrals over clearly
specified simplices.
In our situation Eie's theory implies that
\begin{equation}\label{equ:Eie'sDepth2Case}
   Z_2(s;r_1,r_2)=J^2\Big(x^{r_1}y^{r_2}(x+y)^{-s}\Big)+\Cyc{r_1,r_2} J^1\Big(\int_0^{-x} x^{r_1}y^{r_2}(x+y)^{-s} \,dy\Big),
\end{equation}
where for positive integers $a_j$ ($j=1,\dots,m$)
\begin{equation}\label{equ:Eie'sJJ2}
J^m(x_1^{a_1}\dots x_m^{a_m})= \prod_{j=1}^m (-1)^{a_j}\beta_{a_j+1}.
\end{equation}
For the integral we may use substitution $y=-xt$ to get
\begin{align}\label{equ:eq00}
\int_0^{-x} x^{r_1}y^{r_2}(x+y)^{2n-r_1-r_2-2}\, dy=&
(-1)^{r_2+1} x^{2n-1}\int_0^{1} t^{r_2} (1-t)^{2n-r_1-r_2-2}\,dt \\
=& (-1)^{r_2+1} x^{2n-1}\frac{r_2!(2n-r_1-r_2-2)!}{(2n-r_1-1)!}.
\end{align}
Applying the operator $J$ we get
\begin{equation*}
  Z_2(s;r_1,r_2)=\sum_{a+b=-s} \frac{(-s)!}{a!b!}
\beta_{a+r_1+1}\beta_{b+r_2+1}
+\Cyc{r_1,r_2}\frac{(-1)^{r_2}r_2!(-s)!}{(2n-r_1-1)!}\beta_{2n}.
\end{equation*}
By comparing this with Eq.~\eqref{equ:Eie'sDepth2Case} we have
\begin{align*}
\ & \sum_{k=r_1+1}^{2n-r_2-1} \frac{(2n-r_1-r_2-2)!}{(k-r_1-1)!(2n-k-r_2-1)!}
    \frac{B_{k}}{k}\frac{B_{2n-k}}{2n-k} \\
=&\left(\frac{-r_1!r_2!}{(r_1+r_2+1)!}-\Cyc{r_1,r_2}\frac{ (-1)^{r_2}r_2!(2n-r_1-r_2-2)!}{(2n-r_1-1)!}\right)\frac{B_{2n}}{2n} \\
& -\Cyc{r_1,r_2}(-1)^{r_1}\sum_{k=r_1+1}^{r_1+r_2+1}\binom{r_2}{k-r_1-1}
 \frac{B_{k} }{k}  \frac{B_{2n-k}}{2n-k}.
\end{align*}
The theorem follows immediately.
\end{proof}

\begin{cor}\label{cor:myGenBernoulliWtSumDepth2r1ner2}
For all nonnegative integers $r_1$, $r_2$ and $n\ge r_1+r_2+2$ we have
\begin{align*}
\ & \sum_{k=\lceil (r_1+r_2)/2\rceil+1}^{n-\lceil (r_1+r_2)/2\rceil-1}
    \frac{B_{2k}}{(2k)!}\frac{B_{2j}}{(2j)!} \binom{2k-1}{r_1}  \binom{2j-1}{r_2}  \qquad (\text{here }j=n-k)\\
=&-\left(\binom{2n-1}{r_1+r_2+1}+\Cyc{r_1,r_2} (-1)^{r_2} \binom{2n-1}{r_1}\right)\frac{B_{2n}}{(2n)!} \\
& -\Cyc{r_1,r_2}\sum_{k=\lceil r_1/2\rceil+1}^{\lfloor(r_1+r_2+1)/2\rfloor}
\binom{2k-1}{r_1}\left( (-1)^{r_1}\binom{2j-1}{r_1+r_2+1-2k}
 +  \binom{2j-1}{r_2} \right) \frac{B_{2k} }{(2k)!}  \frac{B_{2j}}{(2j)!}.
\end{align*}
\end{cor}
\begin{proof} By breaking the left sum in Theorem~\ref{thm:myGenBernoulliWtSumDepth2} as
\begin{equation*}
      \sum_{k=r_1+1}^{2n-r_2-1}= \sum_{k=r_1+1}^{r_1+r_2+1}+\sum_{k=r_1+r_2+2}^{2n-r_1-r_2-2}+\sum_{k=2n-r_1-r_2-1}^{2n-r_2-1}
\end{equation*}
we see easily that
\begin{align*}
\ & \sum_{k=\lceil (r_1+r_2)/2\rceil+1}^{n-\lceil (r_1+r_2)/2\rceil-1}
    \frac{B_{2k}}{(2k)!}\frac{B_{2n-2k}}{(2n-2k)!} \prod_{a=1}^{r_1}(2k-a)\prod_{b=1}^{r_2}(2n-2k-b) \\
=&-r_1!r_2!\left(\binom{2n-1}{r_1+r_2+1}+\Cyc{r_1,r_2} (-1)^{r_2} \binom{2n-1}{r_1}\right)\frac{B_{2n}}{(2n)!} \\
& -r_1!r_2!\, \Cyc{r_1,r_2}\sum_{k=\lceil (r_1+1)/2\rceil}^{\lfloor(r_1+r_2+1)/2\rfloor}
\binom{2k-1}{r_1}\left( (-1)^{r_1}\binom{2j-1}{r_1+r_2+1-2k}
 +  \binom{2j-1}{r_2} \right) \frac{B_{2k} }{(2k)!}  \frac{B_{2j}}{(2j)!}.
\end{align*}
Notice that in the last summation if $r_1$ is odd and $k=(r_1+1)/2$
then the corresponding term happens to be zero.
Thus we can improve the lower limit of $k$ from $\lceil (r_1+1)/2\rceil$ to $\lceil r_1/2\rceil+1$.
This finishes the proof of the corollary.
\end{proof}

By Eq.~\eqref{equ:EulerFamousID} and Corollary~\ref{cor:myGenBernoulliWtSumDepth2r1ner2}, we can get the following Corollary
\begin{cor}\label{cor:myGenBernoulliWtSumDepth2r1ner2Zeta}
For all nonnegative integers $r_1$, $r_2$ and $n\ge r_1+r_2+2$ we have
{\small
\begin{align*}
\ & \sum_{k=\lceil (r_1+r_2)/2\rceil+1}^{n-\lceil (r_1+r_2)/2\rceil-1}
   \binom{2k-1}{r_1}  \binom{2(n-k)-1}{r_2}  \zeta(2(n-k)) \zeta(2k)\\
=&\frac12\left(\binom{2n-1}{r_1+r_2+1}+\Cyc{r_1,r_2} (-1)^{r_2} \binom{2n-1}{r_1}\right)\zeta(2n) \\
& -\Cyc{r_1,r_2}\hspace{-.3cm}\sum_{k=\lceil r_1/2\rceil+1}^{\lfloor(r_1+r_2+1)/2\rfloor}
\binom{2k-1}{r_1}\left( (-1)^{r_1}\binom{2(n-k)-1}{r_1+r_2+1-2k}
 +  \binom{2(n-k)-1}{r_2} \right) \zeta(2(n-k))\zeta(2k) .
\end{align*}
}
\end{cor}

In the literature there are many different types of sum formulas for DZV. The following
statement is a kind of weighted sum formula for a product of two Riemann zeta values with
the weight coefficients given by arbitrary polynomials. To derive sum formulas for
DZV one has to symmetrize the coefficient which will be done in the next section.

\begin{thm}\label{thm:myGenWtSumDepth2ZetaPolyCoeff}
Let $F(x,y)\in\Q[x,y]$ be a polynomial of degree $d$. Then for every positive integer $n\ge d+2$ we have
\begin{equation*}
  \sum_{j+k=n}  F(j,k)  \zeta(2j) \zeta(2k) \\
= \sum_{k=0}^{\lfloor(d+1)/2\rfloor} K_{F,k}(n)\zeta(2k)\zeta(2n-2k),
\end{equation*}
where $K_{f,k}(x)$ is a polynomial in $x$ depending only on $F$ and $k$ whose degree is at most $d+1$.
\end{thm}

\begin{proof}
It is well-known that for any nonnegative integers $m\ge r$ the Stirling numbers of
the second kind $S(m,r)$ are all rational numbers which can be defined by
\begin{equation}\label{equ:Stirling}
x^m= \sum_{r=0}^m r! S(m,r)\binom{x}{r}.
\end{equation}
Thus for any nonnegative integers $r_1$ and $r_2$ we have
\begin{multline*}
j^{m_1} k^{m_2}\zeta(2j) \zeta(2k)=\frac{1}{2^{m_1+m_2}}\sum_{r_1=0}^{m_1} \sum_{r_2=0}^{m_2} r_1! r_2! S(m_1,r_1) S(m_2,r_2)\\
\left(\binom{2j-1}{r_1}+\binom{2j-1}{r_1-1}\right)\left(\binom{2k-1}{r_2}+\binom{2k-1}{r_2-1}\right) \zeta(2j) \zeta(2k).
\end{multline*}
The theorem now follows from Corollary~\ref{cor:myGenBernoulliWtSumDepth2r1ner2Zeta}.
\end{proof}

\begin{eg}
We can obtain the following weighted sum formulas applying Corollary~\ref{cor:myGenBernoulliWtSumDepth2r1ner2Zeta}:
{\allowdisplaybreaks
\begin{align}\label{equ:Prod2RiemannZetaFactor=j}
\sum_{\substack{j+k=n\\ j,k\ge 1}}  j  \zeta(2l)\zeta(2k)= & \frac{n(2n+1)}4 \zeta(2n)          ,\\
\sum_{\substack{j+k=n\\ j,k\ge 1}}  j^2 \zeta(2l)\zeta(2k)= &
    \frac{n(2n+1)(4n+1)}{24} \zeta(2n)-\frac{2n-3}{2} \zeta(2) \zeta(2n-2)           ,\notag \\
\label{equ:Prod2RiemannZetaFactor=j3}
\sum_{\substack{j+k=n\\ j,k\ge 1}}  j^3 \zeta(2l)\zeta(2k)= &
    \frac{n^2(2n+1)^2}{16} \zeta(2n) -\frac{3n(2n-3)}{4} \zeta(2) \zeta(2n-2)        ,\\
\sum_{\substack{j+k=n\\ j,k\ge 1}}  j^4 \zeta(2l)\zeta(2k)= &
    \frac{n(2n+1)(4n+1)(12n^2+6n-1) \zeta(2n)}{480}                     \notag\\
    -& \frac{(2n-3)(8n^2-6n+5)}{8}\zeta(2) \zeta(2n-2)
    -\frac{3(2n-5)}{2} \zeta(4)\zeta(2n-4)        .\notag
\end{align}}
\end{eg}

\section{Weighted sum formulas for double zeta values}
\label{sec:double}
In this section we apply results from the last section to give weighted sum formula
for DZVs $\zeta(2j,2k)$ for fixed $j+k$ when the weight factors are arbitrary
symmetric polynomials in $j$ and $k$. First, by setting $r_1=r_2$ in
Corollary~\ref{cor:myGenBernoulliWtSumDepth2r1ner2} we obtain immediately
\begin{prop}\label{prop:myGenBernoulliWtSumDepth2}
For all nonnegative integers $r$ and $n\ge 2r+2$ we have
\begin{align*}
\ & \sum_{k=r+1}^{n-r-1} \frac{(2n-2r-2)!}{(2k-r-1)!(2n-2k-r-1)!}
    \frac{B_{2k}}{2k}\frac{B_{2n-2k}}{2n-2k}  \\
=&-\left(\frac{2(-1)^{r}r!(2n-2r-2)!}{(2n-r-1)!}+\frac{(r!)^2}{(2r+1)!}\right)\frac{B_{2n}}{2n}
        \notag \\
&- \sum_{k=\lceil r/2\rceil+1}^r \frac{2}{(2k-r-1)!}
\left\{\frac{ (-1)^rr!}{(2r-2k+1)!}+\frac{(2n-2r-2)!}{(2n-2k-r-1)!}\right\}
 \frac{B_{2k} }{2k}  \frac{B_{2n-2k}}{2n-2k}.
\end{align*}
\end{prop}

\begin{eg}
Taking $r=1$ we recover \cite[Proposition 1]{Eie1996} (notice it
has a typo: $(2n-2)!$ on the left numerator
should be $(2n-4)!$). Taking $r=2$ we recover
\cite[Proposition~4.2.2]{Eie2009}.
Taking $r=3$ in Proposition~\ref{prop:myGenBernoulliWtSumDepth2} we get
\begin{align*}
\ & \sum_{k=4}^{n-4} \frac{(2n-8)!}{(2k-4)!(2n-2k-4)!}
    \frac{B_{2k}}{2k}\frac{B_{2n-2k}}{2n-2k} \notag\\
=&-\frac{(n-6)(2n+1)(2n^2-11n+35)}{140(n-2)(n-3)(2n-5)(2n-7)}\frac{B_{2n}}{2n}
-\frac{2n-11}{6} B_6 B_{2n-6}, \quad n\ge 8.
\end{align*}
\end{eg}

By  Eq.~\eqref{equ:EulerFamousID} and Proposition~\ref{prop:myGenBernoulliWtSumDepth2}, we obtain
\begin{cor}\label{cor:myGenMZVWtSum}
For all  nonnegative integers $r$ and $n\ge 2r+2$ we have
\begin{align}
\ & \sum_{k=r+1}^{n-r-1} \zeta(2k) \zeta(2n-2k) \prod_{\ga=1}^r \Big\{(2k-\ga)(2n-2k-\ga)\Big\} \notag\\
=&\left( (-1)^{r}r!+\frac{(r!)^2}{2(2r+1)!}\prod_{\gb=r+1}^{2r+1} (2n-\gb) \right)
\zeta(2n) \prod_{\ga=1}^r(2n-\ga)             \label{equ:myGenZetaWtSum}     \\
-& \sum_{k=\lceil r/2\rceil+1}^r \prod_{\ga=1}^r(2k-\ga) \prod_{\gb=2k+1}^{2r+1} (2n-\gb)
\left\{\frac{ 2 (-1)^rr!}{(2r-2k+1)!}+\frac{2 (2n-2r-2)!}{(2n-2k-r-1)!}\right\}
\zeta(2k)\zeta(2n-2k) . \notag
\end{align}
\end{cor}

\begin{eg}
When $r=0$ we get \cite[(2.4)]{Nakamura2009} .
When $r=1$ we recover the formula in Eq.~\eqref{equ:Eier=1Result}.
When $r=2$ we find
\begin{align*}
\ & \sum_{j,k\ge 3, j+k=n} (2j-1)(2j-2) (2k-1)(2k-2)\zeta(2j)\zeta(2k) \notag\\
=&\frac{1}{15}(n-1)(4n^2-1)(2n^2-13n+30)\zeta(2n)
- 24(n-2)(2n-5) \zeta(4)\zeta(2n-4), \quad n\ge 4.
\end{align*}
When $r=3$ we get
\begin{align*}
\ & \sum_{j,k\ge 4, j+k=n} (2j-1)(2j-2)(2j-3) (2k-1)(2k-2)(2k-3)\zeta(2j)\zeta(2k) \notag\\
=&\frac{1}{35}(n-6)(2n-3)(n-1)(4n^2-1)(2n^2-11n+35)\zeta(2n) \\
&-240(2n-11)(n-3)(2n-7) \zeta(6)\zeta(2n-6),\quad n\ge 6.
\end{align*}
When $r=4$ we have
\begin{align*}
\ & \sum_{j,k\ge 5, j+k=n} \zeta(2j)\zeta(2k) \prod_{\ga=1}^4 \Big\{(2j-\ga)(2k-\ga)\Big\} \notag\\
=&\frac{4}{315}(n-2)(2n-3)(n-1)(4n^2-1)(4n^4-72n^3+521n^2-1923n+3780)\zeta(2n) \\
&-960(n-3)(n-4)(2n-7)(2n-9)\zeta(6)\zeta(2n-6)\\
&-6720(n-4)(2n-9)(2n^2-25n+81)\zeta(8)\zeta(2n-8),\quad n\ge 8.
\end{align*}
\end{eg}

To prepare for the next theorem concerning DZVs we need the following
combinatorial statement.
\begin{lem} \label{lem:anyDegDepth2combin}
Let $r$ be a nonnegative integer and let $n$ be an integer variable such that $n\ge 2r+1$.
Then as a polynomial in $n$
\begin{multline}
\label{equ:anyDegDepth2combin}
  \gf_r(n)=\binom{2n-1}{2r+1}-2\sum_{k=r+1}^{n-r-1} \binom{2k-1}{r}\binom{2n-2k-1}{r} \\
  =(-1)^r\binom{n-1}{r}+4\sum_{k=1}^{r} \binom{2k-1}{r}\binom{2n-2k-1}{r}
\end{multline}
has degree less than or equal to $r$.
\end{lem}
\begin{proof}
First we prove that for all integer $m>2r$
\begin{equation}\label{equ:lemmaCombin1}
    \binom{m+1}{2r+1}=\sum_{k=0}^m \binom{k}{r}\binom{m-k}{r}.
\end{equation}
Let $S$ be a set of $m+1$ distinct points on a horizonal line.
It is not hard to see there is a one to one correspondence between the following
two operations: (i) choose
$2r+1$ points from $S$ whose middle point is denoted by $P$; (ii) choose a point $P\in S$ and then choose $r$ points from the $k$ points to the left of $P$ and
$r$ points to the right of $P$.
The two sides of \eqref{equ:lemmaCombin1} clearly give the number of choices in (i) and (ii),
respectively.

Second, we see that for all $m>2r$, under the substitution $k\to m-k$, we have
\begin{equation*}
   f(m,r):=\sum_{k=0}^m (-1)^k \binom{k}{r}\binom{m-k}{r}=(-1)^m f(m,r).
\end{equation*}
Therefore  $f(m,r)=0$ for $m$ odd. We now prove that if $m=2n$ is even then
\begin{equation}\label{equ:lemmaCombin2}
 f(2n,r):=\sum_{k=0}^{2n} (-1)^k \binom{k}{r}\binom{2n-k}{r}=(-1)^r\binom{n}{r}
\end{equation}
by induction on $m+r$. Clearly $f(m,0)=1$. Now for all $r>0$ we have
$$\binom{k}{r}=\binom{k-1}{r}+\binom{k-1}{r-1},\quad
\binom{m-k}{r}=\binom{m-k-1}{r}+\binom{m-k-1}{r-1}.$$
Therefore by definition
\begin{align*}
   f(m,r)=& \sum_{k=0}^m (-1)^k \binom{k-1}{r}\binom{m-k-1}{r}+\sum_{k=0}^m (-1)^k  \binom{k-1}{r-1}\binom{m-k-1}{r}\\
   &+ \sum_{k=0}^m (-1)^k \binom{k-1}{r}\binom{m-k-1}{r-1}+ \sum_{k=0}^m (-1)^k \binom{k-1}{r-1}\binom{m-k-1}{r-1}\\
   =&-f(m-2,r)-f(m-2,r-1) +2\sum_{k=0}^m (-1)^k \binom{k-1}{r}\binom{m-k-1}{r-1},
\end{align*}
since the middle two sums are the same by the substitution $k\to m-k$. It is easy to see that
\begin{equation*}
  \binom{k-1}{r}=\sum_{j=1}^{k-r} \binom{k-j-1}{r-1}.
\end{equation*}
Therefore by induction we get
\begin{align*}
   f(2n,r)=& -f(2n-2,r)-f(2n-2,r-1)+2\sum_{k=0}^{2n}\sum_{j=1}^{k-r} (-1)^k  \binom{k-j-1}{r-1} \binom{2n-k-1}{r-1}\\
=&(-1)^r\binom{n-1}{r-1}-(-1)^r\binom{n-1}{r}+2\sum_{j=1}^{2n} \sum_{k=j+r}^{2n} (-1)^k  \binom{k-j-1}{r-1} \binom{2n-k-1}{r-1}\\
=&(-1)^r\binom{n-1}{r-1} -(-1)^r\binom{n-1}{r}-2\sum_{j=1}^{2n} (-1)^j f(2n-2-j,r-1).
\end{align*}
Noticing  $f(2n-2-j,r-1)=0$ for odd $j$ we get by induction
\begin{align*}
   f(2n,r)=&(-1)^r\binom{n-1}{r-1} -(-1)^r\binom{n-1}{r}+2(-1)^r \sum_{j=1}^{n-1} \binom{n-1-j}{r-1}  \\
=&(-1)^r\binom{n-1}{r-1} +(-1)^r \binom{n-1}{r}  \\
=&(-1)^r\binom{n}{r}.
\end{align*}
Combining \eqref{equ:lemmaCombin1}
and \eqref{equ:lemmaCombin2} we see that
\begin{multline*}
   \binom{2n-1}{2r+1}-2\sum_{k=0}^{n}  \binom{2k-1}{r}\binom{2n-2k-1}{r} \\
=\binom{2n-1}{2r+1}-\sum_{k=0}^{2n-2} \binom{k}{r}\binom{2n-2-k}{r}\Big(1-(-1)^k\Big)=(-1)^r\binom{n-1}{r} .
\end{multline*}
This yields the lemma quickly.
\end{proof}

\begin{thm}\label{thm:wtDblZetaProductFormFactor}
For all positive integers $r$ and $n\ge 2r+2$ we have
\begin{equation*}
\sum_{k=r+1}^{n-r-1} \zeta(2k,2n-2k) \prod_{\ga=1}^r \Big\{(2k-\ga)(2n-2k-\ga)\Big\}
=\sum_{k=0}^{r} c_{r,k}(n) \zeta(2k)\zeta(2n-2k),
\end{equation*}
where $\zeta(0)=-1/2$, $c_{r,k}(x)\in \Q[x]$ depend only on $r$ and $k$ and
have degrees less than or equal to $r$. More precisely,
$c_{r,k}=0$ for $1\le k\le \lceil r/2\rceil$,
$$c_{r,0}(n)=r!^2\left(\frac{(-1)^r}{2} \binom{2n-1}{r}+\frac14 \gf_r(n)\right),$$
where $\gf_r(n)$ is defined in Lemma~\ref{lem:anyDegDepth2combin} and
for all $k>\lceil r/2\rceil$
$$
c_{r,k}(n)=-\prod_{\ga=1}^r(2k-\ga) \prod_{\gb=2k+1}^{2r+1} (2n-\gb)
\left\{\frac{(-1)^rr!}{(2r-2k+1)!}+\frac{(2n-2r-2)!}{(2n-2k-r-1)!}\right\}.
$$
\end{thm}
\begin{proof}
We have the stuffle relation
\begin{equation*}
    \zeta(2k)\zeta(2n-2k)=\zeta(2k,2n-2k)+\zeta(2n-2k,2k)+\zeta(2n).
\end{equation*}
Hence the theorem easily follows from Corollary~\ref{cor:myGenMZVWtSum} together with
Lemma~\ref{lem:anyDegDepth2combin}.
\end{proof}

\begin{eg} \label{eg:MoreWtSumFormulas}
For all $n\ge 4$ we have
\begin{align}
\ & \sum_{k=3}^{n-3}  \zeta(2k,2n-2k) \prod_{\ga=1}^2 \Big\{(2n-2k-\ga)(2k-\ga)\Big\} \notag\\
=&\frac{1}{2}(57n^2-279n+366)\zeta(2n)
- 12(n-2)(2n-5) \zeta(4)\zeta(2n-4) . \label{eqU:ga=2MoreWtSumFormulas}
\end{align}
For all $n\ge 6$
\begin{align}
\ & \sum_{k=4}^{n-4}  \zeta(2k,2n-2k) \prod_{\ga=1}^3 \Big\{(2n-2k-\ga)(2k-\ga)\Big\} \notag\\
=&\frac{1}{2}(1005n^3-12222n^2+48243n-62946)\zeta(2n) \notag \\
&- 120(2n-11)(n-3)(2n-7)\zeta(6)\zeta(2n-6) . \label{eqU:ga=3MoreWtSumFormulas}
\end{align}
For all $n\ge 8$
\begin{align}
\ & \sum_{k=5}^{n-5}  \zeta(2k,2n-2k) \prod_{\ga=1}^4 \Big\{(2n-2k-\ga)(2k-\ga)\Big\} \notag\\
=&\frac{1}{2}(
31116n^4-631800n^3+4846020n^2-16543800n+21168864)\zeta(2n) \notag\\
&-480(n-3)(n-4)(2n-7)(2n-9)\zeta(6)\zeta(2n-6)            \label{eqU:ga=4MoreWtSumFormulas} \\
&-3360(n-4)(2n-9)(2n^2-25n+81)\zeta(8)\zeta(2n-8) . \notag
\end{align}
\end{eg}

Another corollary is a result used by Shen and Cai \cite[Lemma~5]{ShenCa2012}
which they derived by some complicated method using Bernoulli numbers. We
can now prove this result rather quickly.
\begin{cor} \label{cor:firstDegFactorMoreWtSumFormulas}
We have for all $n\ge 2$
\begin{equation}\label{equ:firstDegFactorShenCaiLemma5}
 \sum_{k=1}^{n-1} k(n-k) \zeta(2k,2n-2k)
=\frac{n}{16}\zeta(2n)+\frac{2n-3}4 \zeta(2)\zeta(2n-2).
\end{equation}
\end{cor}
\begin{proof}
Let $S$ be the left hand side of Eq.~\eqref{equ:firstDegFactorShenCaiLemma5}.
By expanding $(2k-1)(2n-2k-1)$ in Eq.~\eqref{equ:NakamuraDblZeta2} we see that
\begin{align*}
\ & \sum_{k=1}^{n-1} (4k(n-k)-2n+1)\zeta(2k,2n-2k)\\
=& \frac34 (n-3)\zeta(2n)+
(2n-3)(\zeta(2,2n-2)+\zeta(2n-2,2)) \\
    =&  \frac{3-5n}4\zeta(2n)+(2n-3)\zeta(2)\zeta(2n-2)
\end{align*}
By Eq.~\eqref{equ:GKZDblZetaEven} we have
\begin{align*}
 4S=&(2n-1) \sum_{k=1}^{n-1} \zeta(2k,2n-2k)+\frac{3-5n}4\zeta(2n)+(2n-3)\zeta(2)\zeta(2n-2)\\
=& \frac{n}4\zeta(2n)+(2n-3)\zeta(2)\zeta(2n-2)
\end{align*}
which yields the corollary at once.
\end{proof}

Similarly we can replace the factor $k(n-k)$ in Corollary~\ref{cor:firstDegFactorMoreWtSumFormulas} by
$k^r(n-k)^r$ for any positive integer $r$. We can even generalize the factor to an arbitrary symmetric
function of two variables evaluated at $k$ and $n-k$.
\begin{lem} \label{lem:anyDegDepth2JKonly}
Let $n$ and $r$ be two nonnegative integers such that $n\ge r+1$. Then
\begin{equation}\label{equ:anyDegDepth2}
  \sum_{k=1}^{n-1} k^r(n-k)^r\zeta(2k,2n-2k)
=\sum_{k=0}^{r} C_{r,k}(n) \zeta(2k)\zeta(2n-2k),
\end{equation}
where $C_{r,k}(x)\in \Q[x]$ depend only on $r$ and $k$ and
have degrees less or equal to $r$.
\end{lem}
\begin{proof}
Let
$$g_r(x,y)=\prod_{\ga=1}^r (x-\ga)(y-\ga).$$
Define the following recursive sequence of polynomials:
\begin{equation*}
f_1(x)=x-2n+1,\quad f_{d+1}(x)= f_1(x)f_d(x-2n+1), \quad\forall d\ge 1.
\end{equation*}
It is not hard to see that
\begin{equation}\label{equ:explicitfd(x)}
 f_d(x)=\prod_{j=1}^d (x-2jn+j)= \sum_{j=0}^d a_{d,j}(n) x^j,
\end{equation}
where $a_{d,d}(n)=1$ and all the coefficients $a_{d,j}(n)$ are
polynomials in $n$ with integer coefficients of degree $d-j$.
Now we claim that for all $d\ge 1$ and $x+y=2n$ we have
\begin{equation}\label{equ:g2fJK}
g_d(x,y)=f_d(xy).
\end{equation}
The case $d=1$ is obvious. By induction
\begin{align*}
 g_{r+1}(x,y)=& (x-1)(y-1) g_{r}(x-1,y-1)\\
 =&(xy-2n+1)f_r((x-1)(y-1)) \\
 =&(xy-2n+1)f_r(xy-2n+1)\\
 =&f_{r+1}(xy).
\end{align*}
Hence Eq.~\eqref{equ:g2fJK} is proved. Together with Eq.~\eqref{equ:explicitfd(x)} this implies
 \begin{equation*}
(JK)^r=g_r(J,K)-\sum_{i=0}^{r-1} a_{r,i}(n) (JK)^i
\end{equation*}
for all even numbers $J=2j$ and $K=2k$ with $J+K=2n$. Therefore
\begin{equation*}
 \sum_{\substack{J,K\ {\rm even},\\ J+K=2n}} (JK)^r \zeta(J)\zeta(K)=
 \sum_{\substack{J,K\ {\rm even},\\ J+K=2n}} g_r(J,K) \zeta(J)\zeta(K)-\sum_{i=0}^{r-1} a_{r,i}(n)
\sum_{\substack{J,K\ {\rm even},\\ J+K=2n}} (JK)^i \zeta(J)\zeta(K)
\end{equation*}
By an easy induction on $r$ and the fact that $\deg_n a_{r,j}(n) =r-j$ the lemma now
follows from Theorem~\ref{thm:wtDblZetaProductFormFactor}. Notice in particular that the difference
in the summation range does not bring in any polynomial coefficients of degree less than or equal to $r$.
This completes the proof of the lemma.
\end{proof}

\begin{eg} \label{eg:secondDegFactorShenCaiLemma5}
We have for all $n\ge 2$
{\allowdisplaybreaks
\begin{align}\label{equ:secondDegFactorShenCaiLemma5}
&\ \sum_{k=1}^{n-1} k^2(n-k)^2 \zeta(2k,2n-2k)
=\frac{3}{32}(3n-2)(n-1)\zeta(2n)-\frac{3}{4}\zeta(4)\zeta(2n-4),\\
&\  \sum_{k=1}^{n-1} k^3(n-k)^3 \zeta(2k,2n-2k)\notag\\
=&-\frac{1}{256}n(2n^2-3)\zeta(2n)+\frac{1}{64}(2n-3)(28n^2-48n+21)\zeta(2)\zeta(2n-2)\notag\\
&+\frac{3}{16}(2n-5)(2n^2-25n+35)\zeta(4)\zeta(2n-4)+\frac{45}{8}(2n-7)\zeta(6)\zeta(2n-6), \label{equ:3rdDegFactorShenCaiLemma5}\\
&\  \sum_{k=1}^{n-1} k^4(n-k)^4 \zeta(2k,2n-2k)\notag\\
=&-\frac{1}{1024}n(16n^2-17)\zeta(2n)
+\frac{1}{256}(2n-3)(6n-5)(20n^2-36n+17)\zeta(2)\zeta(2n-2)\notag\\
&+\frac{3}{64}(2n-5)(40n^3-420n^2+1050n-777)\zeta(4)\zeta(2n-4)\notag\\
&-\frac{15}{16}(2n-7)(8n^2-98n+189)\zeta(6)\zeta(2n-6)
-\frac{315}{4}(2n-9))\zeta(8)\zeta(2n-8).\label{equ:4thDegFactorShenCaiLemma5}
\end{align}}
To show Eq.~\eqref{equ:secondDegFactorShenCaiLemma5} we suppose $j+k=n$. It is easy to see that
\begin{equation*}
\prod_{\ga=1}^2 \Big\{(2j-\ga)(2k-\ga)\Big\}=16j^2k^2+4(5-6n)jk+4(2n-1)(n-1).
\end{equation*}
So Eq.~\eqref{equ:secondDegFactorShenCaiLemma5} follows from
Eqs.~\eqref{equ:GKZDblZetaEven}, \eqref{eqU:ga=2MoreWtSumFormulas}, and
\eqref{equ:firstDegFactorShenCaiLemma5} immediately. The other
formulas can be proved similarly.
\end{eg}

\begin{thm} \label{thm:anyDegDepth2}
Let $F(x,y)=F(y,x)\in \Q[x,y]$ be a symmetric polynomial of degree $r$.
Suppose $d=\deg_x F(x,y)$. Then for every positive integer $n\ge 2$ we have
\begin{equation*}
  \sum_{k=1}^{n-1} F(k,n-k) \zeta(2k,2n-2k)
=\sum_{k=0}^{\lfloor r/2\rfloor} c_{F,k}(n) \zeta(2k)\zeta(2n-2k),
\end{equation*}
where $c_{F,k}(x)\in \Q[x]$ depends only on $k$ and $F$ and
has degrees less than or equal to $d$.
\end{thm}
\begin{proof}
It is a well-known fact that any symmetric polynomial $F(x,y)\in \Q[x,y]$ of degree $r$
is a linear combination of symmetric binomials $x^d y^{r-d}+x^{r-d} y^d$ for $d=\lceil r/2\rceil,\dots,r$. Let $\gs_1=x+y$, $\gs_2=xy$ and $a=r-d$. By induction on the difference $d-a$ it is easy to show that
\begin{equation*}
 x^d y^a+x^a y^d=\gs_1^{d-a}\gs_2^a+\sum_{j=a+1}^{\lfloor r/2\rfloor} c_j  \gs_1^{r-2j}\gs_2^j
\end{equation*}
for some $c_j\in \Z$. Hence the theorem follows from Lemma~\ref{lem:anyDegDepth2JKonly} immediately.
\end{proof}

\section{Weighted sum formulas for triple zeta values}
\label{sec:triple}
In this section we will derive weighted sum formula
for triple zeta values $\zeta(2i,2j,2k)$ for fixed $i+j+k$ when the weight factors are arbitrary
symmetric polynomials in $i$, $j$ and $k$. Similar to the double zeta case we will
first consider weighted sum formula for triple products of Bernoulli numbers which leads to
weighted sum formula for triple products of Riemann zeta values. Then we may apply the stuffle
relations to derive the desired weighted sum formula for triple zeta values.

We begin with the following triple sum:
\begin{equation*}
    Z_3(s;r_1,r_2,0)=\sum_{m_1,m_2,m_3=1}^\infty m_1^{r_1}   m_2^{r_2}(m_1+m_2+m_3)^{-s}.
\end{equation*}
By the substitution $m\to m-m_1-m_2$ we see easily that
 \begin{align}
Z_3(s;r_1,r_2,0)=&\sum_{m=1}^\infty\sum_{m_1=1}^{m-1}\sum_{m_2=1}^{m-m_1-1}m_1^{r_1}m_2^{r_2}m^{-s} \notag\\
     =&\sum_{m=1}^\infty m^{-s}\sum_{m_1=1}^{m-1}m_1^{r_1}\sum_{m_2=1}^{m-m_1-1}m_2^{r_2} \notag\\
     =&\sum_{m=1}^\infty m^{-s}\sum_{m_1=1}^{m-1}m_1^{r_1}\left(\sum_{k=0}^{r_2}\binom{r_2+1}{ k}\frac{(-1)^kB_k}{r_2+1}(m-m_1)^{r_2+1-k}-(m-m_1)^{r_2}\right) \notag\\
     =&\sum_{k=0}^{r_2}\binom{r_2+1}{ k}\frac{(-1)^kB_k}{r_2+1} \sum_{m=1}^\infty m^{-s}\sum_{m_1=1}^{m-1}
     m_1^{r_1}(m-m_1)^{r_2+1-k}-Z_2(s;r_1,r_2) \notag\\
     =&\sum_{k=0}^{r_2}\binom{r_2+1}{
     k}\frac{(-1)^kB_k}{r_2+1}Z_2(s;r_1,r_2+1-k)-Z_2(s;r_1,r_2).\label{equ:withLastZ2}
\end{align}
By Proposition~\ref{prop:Z2general}, the first sigma sum of the above
equation equals
\begin{align}
&\sum_{k=0}^{r_2}\binom{r_2+1}{ k}\frac{(-1)^kB_k}{r_2+1}\tau(r_1,r_2+1-k)\zeta(s+k-r_1-r_2-2)\label{equ:Z3r1r201}\\
     +&\sum_{k=0}^{r_2}\binom{r_2+1}{ k}\frac{(-1)^{r_2+1}B_k}{r_2+1}\sum_{l=r_2+2-k}^{r_1+r_2+2-k}\binom{r_1}{l+k-r_2-2}\beta_l \zeta(s+k+l-r_1-r_2-2)\label{equ:Z3r1r202}\\
     +&\sum_{k=0}^{r_2}\binom{r_2+1}{ k}\frac{(-1)^{r_1+k}B_k}{r_2+1}\sum_{l=r_1+1}^{r_1+r_2+2-k} \binom{r_2+1-k}{     l-r_1-1}\beta_{l}\zeta(s+k+l-r_1-r_2-2).\label{equ:Z3r1r203}
\end{align}
For the first part of the above equation, breaking away the term for $k=0$ we get
\begin{equation*}
\eqref{equ:Z3r1r201} 
=\frac{r_1!r_2!}{(r_1+r_2+2)!}\zeta(s-r_1-r_2-2)
+\sum_{k=1}^{r_2}\frac{(-1)^k r_1!r_2!}{(r_1+r_2-k+2)!}\frac{B_k}{k!}\zeta(s-r_1-r_2+k-2).
\end{equation*}
For the second part of the equation,
\begin{align*}
\eqref{equ:Z3r1r202}=&-\frac{(-1)^{r_2}}{r_2+1}\sum_{j=r_2+2}^{r_1+r_2+2}\binom{r_1}{j-r_2-2}\frac{B_j}{j} \zeta(s+j-r_1-r_2-2)\\
&-\sum_{j=r_2+2}^{r_1+r_2+2}(-1)^{r_2}\sum_{k=1}^{r_2}\binom{r_2}{k-1}\binom{r_1}{j-r_2-2}\frac{B_kB_{j-k}}{k(j-k)}
\zeta(s+j-r_1-r_2-2).
\end{align*}
For the third part of the equation,
\begin{align*}
\eqref{equ:Z3r1r203}=&\sum_{j=r_1+1}^{r_1+r_2+2}\frac{(-1)^{r_1}}{r_2+1}\binom{r_2+1}{
j-r_1-1}\frac{B_j}{j}\zeta(s+j-r_1-r_2-2) \\
&+\sum_{k=1}^{r_2}\sum_{j=r_1+1+k}^{r_1+r_2+2}(-1)^{r_1+k}\binom{r_2}{k-1}\binom{r_2+1-k}{j-k-r_1-1}\frac{B_kB_{j-k}}{k(j-k)}
\zeta(s+j-r_1-r_2-2).
\end{align*}
Setting $s=r_1+r_2+3-2n$ in the above equations and using Proposition~\ref{prop:Z2general}
and Eq.~\eqref{equ:EulerFamousID}, from Eq.~\eqref{equ:withLastZ2} we can get
{\allowdisplaybreaks\begin{align}
&Z_3(r_1+r_2+3-2n;r_1,r_2,0)+Z_2(r_1+r_2+3-2n;r_1,r_2)\notag\\
=&-\frac{r_1!r_2!}{(r_1+r_2+2)!}\frac{B_{2n}}{2n}
-\sum_{k=1}^{r_2}\frac{(-1)^k r_1!r_2!}{(r_1+r_2-k+2)!}\frac{B_k}{k!}\frac{B_{2n-k}}{2n-k}\notag\\
+&\sum_{k=0}^{r_2}\sum_{j=r_2+2}^{r_1+r_2+2} \frac{(-1)^{r_2}}{r_2+1} \binom{r_2+1}{k}\binom{r_1}{j-r_2-2}\frac{B_kB_{j-k}}{j-k}
\frac{B_{2n-j}}{2n-j}\notag\\
-&\sum_{k=0}^{r_2}\sum_{j=r_1+1+k}^{r_1+r_2+2} \frac{(-1)^{r_1+k}}{r_2+1}  \binom{r_2+1}{k}\binom{r_2+1-k}{j-k-r_1-1}\frac{B_kB_{j-k}}{j-k}
\frac{B_{2n-j}}{2n-j}.\label{equ:WtSum3}
\end{align}}

On the other hand, we can compute the value of $Z_3(s;r_1,r_2,r_3)$
using some integrals:
\begin{align*}
   Z_3(s;r_1,r_2,r_3)=&J^3\Big(x_1^{r_1}x_2^{r_2}x_3^{r_3}(x_1+x_2+x_3)^{-s}\Big)\\
   &+\sum_{{\rm cyc}(r_1,r_2,r_3)} J^2\Big(\int_0^{-x_1-x_2} x_1^{r_1}x_2^{r_2}x_3^{r_3}(x_1+x_2+x_3)^{-s}
   \,dx_3\Big)\\
   &+\sum_{{\rm cyc}(r_1,r_2,r_3)} J^1\Big(\int_0^{-x_1}\int_0^{-x_1-x_2} x_1^{r_1}x_2^{r_2}x_3^{r_3}(x_1+x_2+x_3)^{-s}
   \,dx_3dx_2\Big).
\end{align*}
For $J^3$ we use \eqref{equ:Eie'sJJ2} to get
\begin{align*}
&J^3\Big(\sum_{i+j+k=-s}\binom{-s}{i,j,k}x_1^{r_1+i}x_2^{r_2+j}x_3^{r_3+k}\Big) \\
=&\sum_{i+j+k=-s}\binom{-s}{i,j,k}(-1)^{r_1+r_2+r_3-s}\beta_{r_1+i+1}\beta_{r_2+j+1}\beta_{r_3+k+1}.
\end{align*}
For the integral in $J^2$ we use the substitution $x_3=-(x_1+x_2)t$ and
Eq.~\eqref{equ:Eie'sJJ2} to get
{\allowdisplaybreaks
\begin{align*}
&J^2\Big(\int_0^{-x_1-x_2}
x_1^{r_1}x_2^{r_2}x_3^{r_3}(x_1+x_2+x_3)^{-s}
   \,dx_3\Big)\\
=&J^2\Big(\int_0^1
x_1^{r_1}x_2^{r_2}(-x_1-x_2)^{r_3}t^{r_3}(x_1+x_2)^{-s}(1-t)^{-s}(-x_1-x_2)
   \,dt\Big)\\
=&J^2\Big((-1)^{r_3+1}
x_1^{r_1}x_2^{r_2}(x_1+x_2)^{r_3+1-s}\int_0^1t^{r_3}(1-t)^{-s}
   \,dt\Big)\\
=&J^2\Big(
x_1^{r_1}x_2^{r_2}(x_1+x_2)^{r_3+1-s}\Big)(-1)^{r_3+1}\tau(r_3,-s)\\
=&\sum_{i+j=r_3+1-s}\binom{r_3+1-s}{i,j}(-1)^{r_1+r_2-s}\beta_{r_1+i+1}\beta_{r_2+j+1}\tau(r_3,-s).
\end{align*}}
So
\begin{align*}
&J^2\Big(\int_0^{-x_1-x_3}
x_1^{r_1}x_2^{r_2}x_3^{r_3}(x_1+x_2+x_3)^{-s}
   \,dx_2\Big)\\
=&\sum_{i+j=r_2+1-s}\binom{r_2+1-s}{i,j}(-1)^{r_1+r_3-s}\beta_{r_1+i+1}\beta_{r_3+j+1}\tau(r_2,-s),
\end{align*}
and
\begin{align*}
&J^2\Big(\int_0^{-x_2-x_3}
x_1^{r_1}x_2^{r_2}x_3^{r_3}(x_1+x_2+x_3)^{-s}
   \,dx_2\Big)\\
=&\sum_{i+j=r_1+1-s}\binom{r_1+1-s}{i,j}(-1)^{r_2+r_3-s}\beta_{r_2+i+1}\beta_{r_3+j+1}\tau(r_1,-s).
\end{align*}
For the integral in $J^1$ we use substitution $x_3=-(x_1+x_2)t_1,$
and \eqref{equ:eq00} to get
{\allowdisplaybreaks
\begin{align*}
&J^1\Big(\int_0^{-x_1}\int_0^{-x_1-x_2}
x_1^{r_1}x_2^{r_2}x_3^{r_3}(x_1+x_2+x_3)^{-s}
   \,dx_3dx_2\Big)\\
=&J^1\Big(\int_0^{-x_1}(-1)^{r_3+1}
x_1^{r_1}x_2^{r_2}(x_1+x_2)^{r_3+1-s}dx_2\int_0^1t^{r_3}(1-t)^{-s}
   \,dt_1\Big)\\
=&(-1)^{r_1-s+r_3+1}\tau(r_2,-s+r_3+1)\beta_{r_1+r_2-s+r_3+1+2}(-1)^{r_3+1}\tau(r_3,-s)\\
=&(-1)^{r_1-s}\tau(r_2,-s+r_3+1)\tau(r_3,-s)\beta_{r_1+r_2+r_3-s+3}\\
=&(-1)^{r_1-s}\frac{r_2!r_3!(-s)!}{(r_2+r_3-s+2)!}\beta_{r_1+r_2+r_3-s+3}.
\end{align*}}
Combining the above equations by setting $r_3=0,s=r_1+r_2+3-2n$ we find
{\allowdisplaybreaks
\begin{align*}
  \ &Z_3(r_1+r_2+3-2n;r_1,r_2,0)\\
   =&-\sum_{i+j+k=2n-3-r_1-r_2}\binom{2n-3-r_1-r_2}{i,j,k}\beta_{r_1+i+1}\beta_{r_2+j+1}\beta_{k+1}\\
   &-\sum_{i+j=2n-2-r_1-r_2}\binom{2n-2-r_1-r_2}{i,j}\beta_{r_1+i+1}\beta_{r_2+j+1}\tau(0,2n-3-r_1-r_2)\\
   &-\Cyc{r_1,r_2} (-1)^{r_2}\sum_{i+j=2n-2-r_1}\binom{2n-2-r_1}{i,j}\beta_{r_1+i+1}\beta_{j+1}\tau(r_2,2n-3-r_1-r_2)\\
&+(-1)^{-s}\frac{r_1!r_2!(-s)!}{(r_1+r_2-s+2)!}\beta_{2n}
+\Cyc{r_1,r_2} (-1)^{r_1-s}\frac{r_2!r_3!(-s)!}{(r_2+r_3-s+2)!}\beta_{2n}\\
   =&-\sum_{i+j+k=2n-3-r_1-r_2}\binom{2n-3-r_1-r_2}{i,j,k}\frac{B_{r_1+i+1}B_{r_2+j+1}B_{k+1}}{(r_1+i+1)(r_2+j+1)(k+1)}\\
   &-\frac{1}{2n-2-r_1-r_2}\sum_{i+j=2n-2-r_1-r_2}\binom{2n-2-r_1-r_2}{i,j}\frac{B_{r_1+i+1}}{r_1+i+1}\frac{B_{r_2+j+1}}{r_2+j+1}\\
   &-\Cyc{r_1,r_2} \frac{(-1)^{r_2}r_2!(2n-3-r_1-r_2)!}{(2n-2-r_1)!}\sum_{i+j=2n-2-r_1}\binom{2n-2-r_1}{i,j}\frac{B_{r_1+i+1}}{r_1+i+1}\frac{B_{j+1}}{j+1}\\
  &-\left((-1)^{r_1+r_2}\frac{r_1!r_2!(2n-3-r_1-r_2)!}{(2n-1)!}+\Cyc{r_1,r_2} (-1)^{r_2}\frac{r_2!(2n-3-r_1-r_2)!}{(2n-1-r_1)!}\right)\frac{B_{2n}}{2n}.
\end{align*}
}
By comparing this with Eq.~\eqref{equ:WtSum3} we obtain
{\allowdisplaybreaks
\begin{align*}
&(2n-3-r_1-r_2)!\sum_{\substack{i+j+k=2n\\ i,j,k\ge 1}}\frac{B_{i}B_{j}B_{k}}{i!j!k!}
    \left\{\prod_{a=1}^{r_1}(i-a)\prod_{b=1}^{r_2}(j-b)\right\} \\
   &+(2n-3-r_1-r_2)!\sum_{\substack{i+j=2n\\ i,j\ge 1}}\frac{B_{i}}{i!}\frac{B_{j}}{j!}
    \left\{\prod_{a=1}^{r_1}(i-a)\prod_{b=1}^{r_2}(j-b)\right\}   \\
   &+\Cyc{r_1,r_2}(-1)^{r_2}r_2!(2n-3-r_1-r_2)!\sum_{\substack{i+j=2n\\ i,j\ge 1}}\frac{B_{i}}{i!}\frac{B_{j}}{j!}
    \left\{\prod_{a=1}^{r_1}(i-a) \right\}   \\
   &+\left((-1)^{r_1+r_2}\frac{r_1!r_2!(2n-3-r_1-r_2)!}{(2n-1)!}
    +\Cyc{r_1,r_2} (-1)^{r_1}\frac{r_1!(2n-3-r_1-r_2)!}{(2n-1-r_2)!}\right)\frac{B_{2n}}{2n} \\
=&\frac{r_1!r_2!}{(r_1+r_2+2)!}\frac{B_{2n}}{2n}
+\sum_{k=1}^{r_2}\frac{(-1)^k r_1!r_2!}{(r_1+r_2-k+2)!}\frac{B_k}{k!}\frac{B_{2n-k}}{2n-k} \\
&- \sum_{k=0}^{r_2}\sum_{j=r_2+2}^{r_1+r_2+2} \frac{(-1)^{r_2}}{r_2+1} \binom{r_2+1}{k}\binom{r_1}{j-r_2-2}\frac{B_kB_{j-k}}{j-k}
\frac{B_{2n-j}}{2n-j} \\
&+\sum_{k=0}^{r_2}\sum_{j=r_1+1+k}^{r_1+r_2+2} \frac{(-1)^{r_1+k}}{r_2+1} \binom{r_2+1}{k}\binom{r_2+1-k}{j-k-r_1-1}\frac{B_kB_{j-k}}{j-k}
\frac{B_{2n-j}}{2n-j}+Z_2(r_1+r_2+3-2n;r_1,r_2).
\end{align*}}
By Proposition~\ref{prop:Z2general} we have
\begin{multline}\label{equ:Z2term}
Z_2(r_1+r_2+3-2n;r_1,r_2)
=-\Cyc{r_1,r_2} (-1)^{r_1}\sum_{k=r_1+1}^{r_1+r_2+1}\binom{r_2}{k-r_1-1} \beta_k \beta_{2n-k-1}\\
=(\delta_{r_1,0} +\delta_{r_2,0}) \frac{B_{2n-2}}{4n-4}.
\end{multline}
Using Theorem \ref{thm:myGenBernoulliWtSumDepth2} and \eqref{equ:Z2term} we can obtain the next result.
\begin{thm}\label{thm:ProdBernDepth3}
For any nonnegative integers $r_1$ and $r_2$ and positive integer $n\ge r_1+r_2=2$ we have
{\allowdisplaybreaks
\begin{align*} &     \frac{1}{r_1!r_2!}\sum_{\substack{i+j+k=2n\\ i,j,k\ge 1}}\frac{B_{i}B_{j}B_{k}}{i!j!k!}
\left\{\prod_{a=1}^{r_1}(i-a)\prod_{b=1}^{r_2}(j-b)\right\}  \\
=& \Cyc{r_1,r_2}
 \left(\frac{(-1)^{r_1+r_2} B_{r_1+1}B_{2n-r_1-1}}{(r_1+1)!(2n-r_1-1)!}
+\frac{\delta_{r_1,0}}{2} \binom{2n-3}{r_2} \frac{B_{2n-2}}{(2n-2)!} \right)\\
& + \left(\binom{2n}{r_1+r_2+2}+\Cyc{r_1,r_2}(-1)^{r_2}\binom{2n}{r_1+1} +(-1)^{r_1+r_2}\right)\frac{B_{2n}}{(2n)!} \\
& + \Cyc{r_1,r_2}(-1)^{r_1} \sum_{k=r_1+1}^{r_1+r_2+1} \binom{k-1}{r_1}\binom{2n-k-1}{r_2+r_1+1-k}
 \frac{B_{k} }{k!}  \frac{B_{2n-k}}{(2n-k)!}\\
&+\Cyc{r_1,r_2}(-1)^{r_2}  \sum_{k=1}^{r_1+1} \binom{2n-1-k}{r_1+1-k}
 \frac{B_{k} }{k!}  \frac{B_{2n-k}}{(2n-k)!}\\
&+\sum_{k=1}^{r_2} (-1)^k  \binom{2n-1-k}{r_1+r_2+2-k}\frac{B_k}{k!}\frac{B_{2n-k}}{(2n-k)!}\\
&-\sum_{k=0}^{r_2}\sum_{j=r_2+2}^{r_1+r_2+1}  (-1)^{r_2}  \binom{j-1-k}{r_2+1-k}\binom{2n-1-j}{r_1+r_2+2-j} \frac{B_kB_{j-k}B_{2n-j}}{k!(j-k)!(2n-j)!} \\
&+\sum_{k=0}^{r_2}\sum_{j=r_1+1+k}^{r_1+r_2+1}  (-1)^{r_1+k}  \binom{j-1-k}{r_1}\binom{2n-1-j}{r_1+r_2+2-j} \frac{B_kB_{j-k}B_{2n-j}}{k!(j-k)!(2n-j)!} .
\end{align*}}
\end{thm}
\begin{proof}
We only need to show that all the terms with $j=r_1+r_2+2$ cancel out. Indeed, for such a term to be nonzero both $k$ and $r_1+r_2$ must be even. Hence the the sign $(-1)^{r_2} $ in the penultimate sum and the sign $(-1)^{r_1+k}$ in the last sum are the same. Moreover
\begin{equation*}
     \binom{j-1-k}{r_2+1-k}=\binom{j-1-k}{r_1}
\end{equation*}
when $j=r_1+r_2+2$. Therefore all these terms cancel out.
\end{proof}

\begin{eg} \label{eg:BernDepth3}
Set $\sum=\sum_{\substack{i+j+k=2n\\ i,j,k\ge 1}}$.
For small $r_1$ and $r_2$ we may use Theorem~\ref{thm:ProdBernDepth3} for $n\ge r_1+r_2+2$
and direct computation for small $n$ to verify the following identities which are valid for all $n\ge 2$:
{\allowdisplaybreaks
\begin{align*} 
\sum \frac{B_{i}B_{j}B_{k}}{i!j!k!}
    =& \binom{2n+2}{2}\frac{B_{2n}}{(2n)!}+\frac{B_{2n-2}}{(2n-2)!},\\
\sum i\frac{B_{i}B_{j}B_{k}}{i!j!k!}
    =&\binom{2n+2}{3}\frac{B_{2n}}{(2n)!}+\frac{2n}{3}\frac{B_{2n-2}}{(2n-2)!}  ,\\
\sum ij\frac{B_{i}B_{j}B_{k}}{i!j!k!}
    =&\binom{2n+2}{4} \frac{B_{2n}}{(2n)!}-\frac{2n^2-19n+12}{12}\frac{B_{2n-2}}{(2n-2)!},\\
\sum i^2j\frac{B_{i}B_{j}B_{k}}{i!j!k!}
    =&\frac{4n+1}{5} \binom{2n+2}{4}\frac{B_{2n}}{(2n)!}
    +\frac{10n^2-11n+6}{12}\frac{B_{2n-2}}{(2n-2)!}-\frac{2n-5}{60}\frac{B_{2n-4}}{(2n-4)!},\\
\sum ijk\frac{B_{i}B_{j}B_{k}}{i!j!k!}
=&\binom{2n+2}{5}\frac{B_{2n}}{(2n)!}
    -\frac{2n^3-9n^2+n+6}{6}\frac{B_{2n-2}}{(2n-2)!}+\frac{2n-5}{30}\frac{B_{2n-4}}{(2n-4)!},\\
\sum i^3\frac{B_{i}B_{j}B_{k}}{i!j!k!}
=&\frac{12n^2+12n+1}{10}\binom{2n+2}{3}\frac{B_{2n}}{(2n)!}  \\
&+  \frac{20n^3-48n^2+35n-6}{6} \frac{B_{2n-2}}{(2n-2)!}
+\frac{2n-5}{30}\frac{B_{2n-4}}{(2n-4)!},\\
\sum i^2j^2\frac{B_{i}B_{j}B_{k}}{i!j!k!}
=&\frac{8n^2+4n+3}{15}\binom{2n+2}{4}\frac{B_{2n}}{(2n)!}  \\
&+ \frac{8n^4-72n^3+232n^2-261n+108}{36}\frac{B_{2n-2}}{(2n-2)!}
-\frac{(7n+3)(2n-5)}{180}\frac{B_{2n-4}}{(2n-4)!} ,\\
\sum i^3j\frac{B_{i}B_{j}B_{k}}{i!j!k!}
=&\frac{4n^2+2n-1}{5}\binom{2n+2}{4}\frac{B_{2n}}{(2n)!}  \\
&+ \frac{32n^3-98n^2+107n-36}{12}\frac{B_{2n-2}}{(2n-2)!}
-\frac{(n-1)(2n-5)}{60}\frac{B_{2n-4}}{(2n-4)!}  ,\\
\sum i^4\frac{B_{i}B_{j}B_{k}}{i!j!k!}
=&\frac{(8n^2+8n-1)(2n+1)}{10}\binom{2n+2}{3}\frac{B_{2n}}{(2n)!}  \\
&+ \frac{40n^4-128n^3+168n^2-119n+36}{6}\frac{B_{2n-2}}{(2n-2)!}
+\frac{(3n-1)(2n-5)}{60}\frac{B_{2n-4}}{(2n-4)!}  .
\end{align*}}
For all $n\ge 3$ we have
\begin{align*} 
\sum i^5\frac{B_{i}B_{j}B_{k}}{i!j!k!}
=&\frac{(8n^2+8n+1)(4n^2+4n-1)}{14}\binom{2n+2}{3}\frac{B_{2n}}{(2n)!}  \\
&+ \frac{80n^5-320n^4+560n^3-520n^2+219n-30}{6}\frac{B_{2n-2}}{(2n-2)!} \\
&+ \frac{40n^4-128n^3+168n^2-119n+36}{6}\frac{B_{2n-4}}{(2n-4)!} \\
&+\frac{(2n-5)(2n^2-6n+7)}{6}\frac{B_{2n-6}}{(2n-6)!}  .
\end{align*}
\end{eg}

\begin{thm}\label{thm:ProdDepth3}
Let $F(x,y,z)\in\Q[x,y,z]$ be a polynomial of degree $d$. Then for every positive integer $n\ge d+2$ we have
\begin{equation}
  \sum_{\substack{i+j+k=n\\ i,j,k\ge 1}}  F(i,j,k)   \zeta(2i) \zeta(2j) \zeta(2k) \\
= \sum_{k=0}^{\lfloor(d+1)/2\rfloor} K_{F,k}(n)\zeta(2k)\zeta(2n-2k),
\end{equation}
where $K_{F,k}(x)$ is a polynomial in $x$ depending only on $F$ and $k$ which can be explicitly given using
Theorem~\ref{thm:ProdBernDepth3}. Moreover, $\deg K_{F,0}(x)=d+2$ and
$\deg K_{F,k}(x)\le d$ for all $k\ge 1$.
\end{thm}
\begin{proof}
We can apply Euler's identity in Eq.~\eqref{equ:EulerFamousID} to Theorem~\ref{thm:ProdBernDepth3} to derive a weighted sum formula for
the triple product of Riemann zeta values. Then we may use Eq.~\eqref{equ:Stirling} to change the weight factors to any monomial of
the form $i^{r_1}j^{r_2}$. Finally we can deduce the theorem by observing that
$i^{r_1}j^{r_2}k^{r_3}=i^{r_1}j^{r_2}(n-i-j)^{r_3}$.
\end{proof}

\begin{eg} \label{eg:TripleZetaDepth3}
Setting $\sum=\sum_{\substack{i+j+k=n\\ i,j,k\ge 1}}$ we have
{\allowdisplaybreaks
\begin{align}
  \sum   \zeta(2i) \zeta(2j) \zeta(2k)
= & \frac{(n+1)(2n+1)}{4}\zeta(2n)-\frac32\zeta(2n-2)\zeta(2),\label{equ:TplProdRiemannZetaFactor=1}\\
  \sum i  \zeta(2i) \zeta(2j) \zeta(2k)
= & \frac{n(n+1)(2n+1)}{12}\zeta(2n)-\frac{n}{2}\zeta(2n-2)\zeta(2),\notag \\
  \sum  ij \zeta(2i) \zeta(2j) \zeta(2k)
= & \frac{n(4n^2-1)(n+1)}{96}\zeta(2n)+\frac{(2n-1)(n-3)}{8}\zeta(2n-2)\zeta(2), \label{equ:TplProdRiemannZetaFactor=ij}\\
  \sum i^2 j  \zeta(2i) \zeta(2j) \zeta(2k)
= & \frac{n(4n+1)(4n^2-1)(n+1)}{960}\zeta(2n)   \notag  \\
&+\frac{(2n-1)(n-3)}{16}\zeta(2n-2)\zeta(2)+\frac{3(2n-5)}{4}\zeta(2n-4)\zeta(4), \notag \\
  \sum ijk  \zeta(2i) \zeta(2j) \zeta(2k)
= & \frac{n(n^2-1)(4n^2-1)}{480}\zeta(2n)  \notag \\
&+\frac{(n-1)(n-3)(2n-1)}{8}\zeta(2)\zeta(2n-2) \notag
+ \frac32(2n-5)\zeta(4)\zeta(2n-4).
\end{align}}
\end{eg}

\begin{thm}\label{thm:MZVDepth3}
Let $F(x,y,z)\in\Q[x,y,z]$ be a symmetric polynomial of degree $d$. Then for every positive integer $n\ge d+2$ we have
\begin{equation}
  \sum_{\substack{i+j+k=n\\ i,j,k\ge 1}}  F(i,j,k)   \zeta(2i,2j,2k) \\
= \sum_{k=0}^{\lfloor(d+1)/2\rfloor} C_{F,k}(n)\zeta(2k)\zeta(2n-2k),
\end{equation}
where $C_{F,k}(x)$ is a polynomial in $x$ depending only on $F$ and $k$ with
$\deg C_{F,k}(x)\le d+2$ for all $k\ge 0$.
\end{thm}
\begin{proof} Let $S_3$ be the permutation group of $\{1,2,3\}$. For any function $f(x_1,x_2,x_3)$ of three variables we set
\begin{equation*}
    \Sym{x_1,x_2,x_3} f(x_1,x_2,x_3)=\sum_{\gs\in S_3} f(x_{\gs(1)},x_{\gs(2)},x_{\gs(3)}).
\end{equation*}
If $i+j+k=n$ then by the stuffle relation we have
\begin{equation*}
2 \zeta(2i) \zeta(2j) \zeta(2k)=2\Sym{i,j,k} \zeta(2i,2j,2k)+\Sym{i,j,k}\zeta(2i+2j)\zeta(2k)-4\zeta(2n).
\end{equation*}
Therefore by setting $\sum=\sum_{\substack{i+j+k=n\\ i,j,k\ge 1}} $ we get
\begin{multline*}
  \sum  F(i,j,k) \zeta(2i) \zeta(2j) \zeta(2k)
  =6 \sum  F(i,j,k) \zeta(2i,2j,2k)\\
  + 3\sum  F(i,j,k) \zeta(2i+2j)\zeta(2k)
  -2\sum  F(i,j,k) \zeta(2n).
\end{multline*}
An easy computation shows that
\begin{align*}
\sum i^aj^b  \zeta(2i+2j)\zeta(2k) &
=   \sum_{l=1}^{n-1}  \sum_{i=1}^{l-1} i^a(l-i)^b  \zeta(2l)\zeta(2n-2l)
=   \sum_{l=1}^{n-1}  f(l) \zeta(2l)\zeta(2n-2l),\\
\sum i^ak^b  \zeta(2i+2j)\zeta(2k) &
=   \sum_{l=1}^{n-1}  \sum_{i=1}^{l-1} i^a(n-l)^b  \zeta(2l)\zeta(2n-2l)
=   \sum_{l=1}^{n-1}  g(l,n) \zeta(2l)\zeta(2n-2l),\\
\sum  i^aj^b \zeta(2n) &
= \sum_{l=1}^{2n-2}  \sum_{i=1}^{l-1} i^a(l-i)^b \zeta(2n)
    = \sum_{l=1}^{n-1}  f(l) \zeta(2n)= h(n) \zeta(2n),
\end{align*}
where $f(x),g(x,y),h(x)\in\Q[x,y]$ are polynomials of degree at most $a+b+1$, $a+b+1$
and $a+b+2$, respectively.
Hence the theorem follows from
Theorem~\ref{thm:myGenWtSumDepth2ZetaPolyCoeff} quickly.
\end{proof}

\begin{eg} \label{eg:MZVDepth3}
Let $e_m(x,y,z)$ be the $m$-th elementary symmetric polynomial of $x,y,z$. Setting
$\sum=\sum_{\substack{i+j+k=n\\ i,j,k\ge 1}}$ we have
{\allowdisplaybreaks
\begin{align}
\sum \zeta(2i,2j,2k)
    =&\frac58\zeta(2n) -\frac14\zeta(2)\zeta(2n-2),      \label{equ:MZVD3factor=1} \\
\sum (ij+jk+ki) \, \zeta(2i,2j,2k)
    =& \frac{5n}{64}\zeta(2n)+\frac{4n-9}{16}\zeta(2) \zeta(2n-2), \label{equ:MZVD3factor=ij} \\
\sum (i^2+j^2+k^2) \, \zeta(2i,2j,2k)
    =& \frac{5n(4n-1)}{32}\zeta(2n)-\frac{2n^2+4n-9}{8}\zeta(2) \zeta(2n-2), \label{equ:MZVD3factor=i2} \\
\sum \bigg(\Sym{i,j,k}i^2 j\bigg)\, \zeta(2i,2j,2k)
    =& \frac{n(10n-3)}{128}\zeta(2n)
        +\frac{8n^2-18n+3}{32} \zeta(2)\zeta(2n-2) \notag \\
    -&\frac{3(2n-5)}{8}\zeta(4)\zeta(2n-4),   \label{equ:MZVD3factor=i2j}  \\
\sum (i^3+j^3+k^3)\, \zeta(2i,2j,2k)
    =& \frac{n(80n^2-30n+3)}{128}\zeta(2n)
        +\frac{3(2n-5)}{8}\zeta(4)\zeta(2n-4) \notag \\
   -&\frac{8n^3+24n^2-54n+3}{32} \zeta(2)\zeta(2n-2) ,   \label{equ:MZVD3factor=i3}  \\
\sum ijk\, \zeta(2i,2j,2k)
    =&\frac{n}{128}\zeta(2n) -\frac{1}{32} \zeta(2)\zeta(2n-2)
        +\frac{2n-5}{8}\zeta(4)\zeta(2n-4). \label{equ:MZVD3factor=ijk}
\end{align}}
We remark that Eq.~\eqref{equ:MZVD3factor=1} is a result of Shen and Cai \cite{ShenCa2012}.
We now prove Eq.~\eqref{equ:MZVD3factor=ij}. The others can be proved using the same idea.
First we have for all nonnegative integers $a$ and $b$
\begin{multline}\label{equ:expansionTplProdRiemannZeta}
  \sum  i^a j^b \zeta(2i) \zeta(2j) \zeta(2k)
  =  \sum \Sym{i,j,k}i^a j^b \zeta(2i,2j,2k)
  +\sum  i^a j^b \zeta(2i+2j)\zeta(2k) \\
  +\sum   i^a k^b \zeta(2i+2j)\zeta(2k)
  +\sum   i^b k^a \zeta(2i+2j)\zeta(2k)
  -2\sum   i^a j^b \zeta(2n).
\end{multline}
For  Eq.~\eqref{equ:MZVD3factor=ij} we note that $\Sym{i,j,k}ij =2e_2(i,j,k)$ so we need the following:
\begin{align*}
\sum ij\, \zeta(2i+2j)\zeta(2k)=& \sum_{l=1}^{n-1}  \sum_{i=1}^{l-1} i(l-i) \zeta(2l)\zeta(2n-2l)
    = \frac16\sum_{\substack{l+k=n\\ l,k\ge 1}}  (l^3-l) \zeta(2l)\zeta(2k),\\
\sum ik\,  \zeta(2i+2j)\zeta(2k)=& \sum_{l=1}^{n-1}  \sum_{i=1}^{l-1} i(n-l) \zeta(2l)\zeta(2n-2l)
    = \frac12\sum_{\substack{l+k=n\\ l,k\ge 1}} \big((n+1) l^2-l^3-nl \big) \zeta(2l)\zeta(2k),\\
\sum  ij\,  \zeta(2n)=& \sum_{l=1}^{n-1} \sum_{i=1}^{l-1} i(n-l) \zeta(2n)=\binom{n+1}{4}\zeta(2n).
\end{align*}
Similarly, by using Eq.~\eqref{equ:expansionTplProdRiemannZeta} and the symmetric function
\begin{equation*}
\Sym{i,j,k} i^2j=e_1(i,j,k)e_2(i,j,k)-3e_3(i,j,k),
\end{equation*}
we see that to show Eqs.~\eqref{equ:MZVD3factor=i2j} and \eqref{equ:MZVD3factor=i3}  we need the following:
{\allowdisplaybreaks
\begin{align*}
  \sum_{i=1}^{l-1} i^2 (l-i)
=&\frac1{12} (l^4-l^2) , \qquad
  \sum_{i=1}^{l-1} i^3
= \frac1{4} (l^4-l^2) ,\\
  \sum_{i=1}^{l-1} i^2 (n-l)
=& \frac16  l(l-1)(2l-1)(n-l)=\frac16 \Big((2n+3)l^3-(3n+1)l^2+nl-2l^4\Big) ,\\
  \sum_{i=1}^{l-1}  (n-l)^3
=& l^4 -(3n+1)l^3 +3(n^2+n)l^2-(n^3+3n^2)l+n^3 ,\\
  \sum_{i=1}^{l-1} i(n-l)^2
=& \frac12  l(l-1) (n-l)^2 =\frac12 \Big(l^4-(2n+1)l^3+(n^2+2n)l^2-n^2 l\Big) ,\\
\sum  i^2 j =& \frac1{12} \sum_{l=1}^{n-1} (l^4-l^2) =\frac{2n-1}{5} \binom{n+1}{4},\qquad
\sum  i^3=\frac{3(2n-1)}{5} \binom{n+1}{4}.
\end{align*}}
Using Eqs.~ \eqref{equ:Prod2RiemannZetaFactor=j}-\eqref{equ:Prod2RiemannZetaFactor=j3},
\eqref{equ:TplProdRiemannZetaFactor=ij} and \eqref{equ:expansionTplProdRiemannZeta}
we can get Eqs.~\eqref{equ:MZVD3factor=ij},  \eqref{equ:MZVD3factor=i2j} and \eqref{equ:MZVD3factor=i3}.

Now multiplying $(i+j+k)^2=n^2$ on Eq.~\eqref{equ:MZVD3factor=1} and comparing with
Eq.~\eqref{equ:MZVD3factor=ij} we can prove Eq.~\eqref{equ:MZVD3factor=i2} easily.
Similarly, by multiplying $(i+j+k)^2=n^3$ on Eq.~\eqref{equ:MZVD3factor=1} we
can readily deduce Eq.~\eqref{equ:MZVD3factor=ijk} from Eqs.~\eqref{equ:MZVD3factor=i2j} and \eqref{equ:MZVD3factor=i3}.
Exactly the same ideas lead to the following:
{\allowdisplaybreaks
\begin{align*}
&\sum (i^2jk+j^2ki+k^2ij) \, \zeta(2i,2j,2k)
     = \frac{n^2}{128}\zeta(2n) -\frac{n}{32}\zeta(2) \zeta(2n-2)
    +\frac{n(2n-5)}{8}\zeta(4)\zeta(2n-4), \\
&\sum (i^2j^2+j^2k^2+k^2i^2) \, \zeta(2i,2j,2k)
     = -\frac{n(2n-5)}{256}\zeta(2n)\\
&\hskip1cm     +\frac{3(12n^2-30n+17)}{64}\zeta(2) \zeta(2n-2)
    -\frac{4n^2+20n-75}{16}\zeta(4)\zeta(2n-4), \\
&\sum \bigg(\Sym{i,j,k}i^3 j\bigg)\, \zeta(2i,2j,2k)
    = \frac{n(10n^2-3n-5)}{128}\zeta(2n)  \\
& \hskip1cm   +\frac{8n^3-54n^2+95n-51}{32} \zeta(2)\zeta(2n-2)
   -\frac{3(2n^2-15n+25)}{8}\zeta(4)\zeta(2n-4),     \\
&  \sum (i^4+j^4+k^4)\, \zeta(2i,2j,2k)
 =\frac{n(80n^3-40n^2+6n+5)}{128}\zeta(2n) \\
&  \hskip1cm       -\frac{8n^4+32n^3-108n^2+98n-51}{32} \zeta(2)\zeta(2n-2)
        +\frac{3(4n^2-20n+25)}{8}\zeta(4)\zeta(2n-4).
\end{align*}}
Moreover, one checks easily that a suitable linear combination of
the four identities above yields
Eq.~\eqref{equ:MZVD3factor=1} multiplied by $(i+j+k)^4=n^4$ because
\begin{equation*}
(i+j+k)^4=\Sym{i,j,k} \Big(i^4 + 4i^3 j+6i^2j^2+12i^2jk \Big).
\end{equation*}
Finally,
when the weight factors have degree five we have
{\allowdisplaybreaks
\begin{align*}
&\sum  \bigg(\Sym{i,j,k} i^2j^2k \bigg) \, \zeta(2i,2j,2k)
  = \frac{n}{256}\zeta(2n)+\frac{2n^2-6n+3}{64}\zeta(2)\zeta(2n-2)\\
&\hskip1cm +\frac{(5n-9)(2n-5)}{16}\zeta(4)\zeta(2n-4)-\frac{3(2n-7)}{8}\zeta(6)\zeta(2n-6),\\
&\sum  \bigg(\Sym{i,j,k} i^3jk \bigg) \,\zeta(2i,2j,2k)
  = \frac{n(n-1)(n+1)}{128}\zeta(2n)-\frac{3(n-1)^2}{32}\zeta(2)\zeta(2n-2)\\
&\hskip1cm +\frac{(2n-5)(n^2-5n+9)}{8}\zeta(4)\zeta(2n-4)+\frac{3(2n-7)}{4}\zeta(6)\zeta(2n-6),\\
&\sum  \bigg(\Sym{i,j,k} i^3j^2 \bigg) \,\zeta(2i,2j,2k)
  = -\frac{n(2n^2+1-5n)}{256}\zeta(2n)+\frac{3(2n-7)}{8}\zeta(6)\zeta(2n-6)\\
&\hskip1cm   +\frac{57n-92n^2+36n^3-3}{64}\zeta(2)\zeta(2n-2)
-\frac{(2n-5)(2n^2+20n-9)}{16}\zeta(4)\zeta(2n-4),\\
&\sum  \bigg(\Sym{i,j,k} i^4j \bigg) \,\zeta(2i,2j,2k)
  = \frac{n(20n^3-8n^2-15n+5)}{256}\zeta(2n)\\
&\hskip1cm +\frac{16n^4-144n^3+294n^2-183n+15}{64}\zeta(2)\zeta(2n-2) \\
&\hskip1cm -\frac{(2n-5)(8n^2-70n+45)}{16}\zeta(4)\zeta(2n-4)-\frac{15(2n-7)}{8}\zeta(6)\zeta(2n-6),\\
&\sum  (i^5+j^5+k^5)\,\zeta(2i,2j,2k)
  = \frac{5n(32n^4-20n^3+4n^2+5n-1)}{256}\zeta(2n)\\
&\hskip1cm -\frac{16n^5+80n^4-360n^3+490n^2-285n+15}{64}\zeta(2)\zeta(2n-2)\\
&\hskip1cm +\frac{5(2n-5)(2n-9)(2n-1)}{16}\zeta(4)\zeta(2n-4)+\frac{15(2n-7)}{8}\zeta(6)\zeta(2n-6).
\end{align*}}
We may check the consistency by using the identity
\begin{equation*}
(i+j+k)^5=\Sym{i,j,k} \Big(i^5 + 5i^4 j+10i^3j^2+20i^3jk+30i^2j^2k \Big).
\end{equation*}

\end{eg}

We end our paper by the following general conjecture which is supported by the above examples in depth 3,
Examples \ref{eg:MoreWtSumFormulas} and Theorem~\ref{thm:anyDegDepth2} in depth 2.
\begin{conj} \label{conj:AnyDegAnyDepth}
Let $F(x_1,\dots,x_m)\in \Q[x_1,\dots,x_m]$ be a symmetric polynomial of total degree $r$.
Suppose $d=\deg_{x_1} F(x_1,\dots,x_m)$. Then for every positive integer $n\ge m$ we have
\begin{align*}
  \sum_{\substack{k_1+\dots+k_m=n\\  k_1,\dots,k_m\ge 1}} F(k_1,\dots,k_m) \zeta(2k_1) \dots\zeta(2k_m)
=\sum_{k=0}^{T} e_{F,k}(n) \zeta(2k)\zeta(2n-2k),\\
  \sum_{\substack{k_1+\dots+k_m=n\\  k_1,\dots,k_m\ge 1}} F(k_1,\dots,k_m) \zeta(2k_1,\dots,2k_m)
=\sum_{k=0}^{T} c_{F,k}(n) \zeta(2k)\zeta(2n-2k),
\end{align*}
where $T=\max\{\lfloor (r+m-2)/2 \rfloor,\lfloor (m-1)/2 \rfloor\}$,
$e_{F,k}(x),c_{F,k}(x)\in \Q[x]$ depend only on $k$ and $F$,
$\deg e_{F,k}(x)\le r-1$ and $\deg c_{F,k}(x)\le d$.
\end{conj}

Notice that $T=\lfloor (m-1)/2 \rfloor$ or $T=\lfloor (r+m-2)/2 \rfloor$ depending on whether $r=0$ or $r>0$.
When $r=d=0$ the second formula of the conjecture follows from the main result in \cite{Hoffman2012} by Hoffman.
When $r=d=1$ then $F(k_1,\dots,k_m)=k_1+,\dots+k_m=n$ so the sum formula reduces to the
case $r=d=0$.

\smallskip

\noindent
{\bf Acknowledgements.} This work is supported by the National Natural Science
Foundation of China (Grant No. 11371178) and the National Science
Foundation of US (Grant No. DMS~1001855 and DMS~1162116).
Part of the work was done while
JZ was visiting the Max-Planck Institute for Mathematics
and the Kavli Institute for Theoretical Physics China
whose support is gratefully acknowledged.

\end{document}